\documentclass[a4paper, 8pt]{amsart}

\usepackage{ mathrsfs }
\usepackage{cite}
\usepackage{geometry}
\usepackage{amscd}
\usepackage{ dsfont }

\DeclareMathOperator\GRT{\text{GRT}}
\DeclareMathOperator\grt{\mathfrak{grt}}
\DeclareMathOperator\DAss{\text{DAss}}

\DeclareMathOperator\cotimes{\hat{\otimes}}
\DeclareMathOperator\bch{\text{BCH}}
\DeclareMathOperator\Lie{\text{Lie}}

\DeclareMathOperator\cFLiexy{\hat{\mathbb{F}}_{\Lie}(x,y)}

\DeclareMathOperator\PhiKZ{\Phi_{\text{KZ}}}
\DeclareMathOperator\PhiAKZ{\Phi_{\overline{\text{KZ}}}}
\DeclareMathOperator\PhiOH{\Phi_{\frac{1}{2}}}
\DeclareMathOperator\PhiAT{\Phi_{\text{AT}}}
\DeclareMathOperator\Badm{B_{\text{adm}}}
\DeclareMathOperator\psiOH{\psi_{\frac{1}{2}}}

\DeclareMathOperator\tpsi{\tilde{\psi}}

\DeclareMathOperator\Aadm{A_{\text{adm}}}

\usepackage[latin1]{inputenc}
\usepackage{enumitem}
\usepackage{nameref}
\usepackage{hyperref}
\usepackage{cite}
 \usepackage[T1]{fontenc}
\usepackage{lastpage, setspace}
\usepackage{graphicx, pst-all}
\usepackage[all]{xy}
\usepackage{amsmath, amsthm, amssymb, tensor}
\usepackage{multicol}
\usepackage{listings}
\usepackage{fancyhdr}
\usepackage{verbatim}
\usepackage{pdflscape}
\usepackage{color}
\usepackage{amsmath,bm}

\usepackage{ mathrsfs }

\makeatletter
\newcommand{\displaybump}{\hbox to \@totalleftmargin{\hfil}}
\makeatother

\psset{linewidth=0.4pt}
\psset{arrowsize=3pt}

\SelectTips{cm}{}
\newdir{ >}{{}*!/-5pt/\dir{>}}

\DeclareMathOperator\ad{ad}

\DeclareMathOperator\frakt{\mathfrak{t}}

\theoremstyle{definition}
\newtheorem{theorem}{Theorem}[section]
\newtheorem*{theorem*}{Theorem}
\newtheorem{lemma}[theorem]{Lemma}
\newtheorem*{lemma*}{Lemma}

\newtheorem{definition}[theorem]{Definition}
\newtheorem{remark}[theorem]{Remark}
\newtheorem*{remark*}{Remark}

\newtheorem{example}[theorem]{Example}
\newtheorem*{example*}{Example}
\newtheorem*{sketch of proof}{Sketch of Proof}
\newtheorem*{question}{Question}

\newtheorem*{fact*}{Fact}

\newtheorem{exampleprop}[theorem]{Example/Proposition}
\newtheorem{conjecture}[theorem]{Conjecture}

\theoremstyle{remark}
\newtheorem*{case1}{1. Case}
\newtheorem*{case2}{2. Case}

\theoremstyle{remark}
\newtheorem*{degree1}{Degree 1}
\newtheorem*{degree2}{Degree 2}

\usepackage[OT2,T1]{fontenc}
\DeclareSymbolFont{cyrletters}{OT2}{wncyr}{m}{n}
\DeclareMathSymbol{\sha}{\mathalpha}{cyrletters}{"58}

\title{On the irrationality of certain coefficients of the Alekseev-Torossian associator}
\author{Matteo Felder}
\begin{document}

\begin{abstract}
We give explicit formulas for the first few coefficients of the Alekseev-Torossian associator and a second Drinfeld associator defined in \cite{Rossi2014}. This is done by analyzing the free and transitive action of the Grothendieck-Teichm\"{u}ller group and its Lie algebra $\grt_1$ on the set of Drinfeld associators. As a result we obtain that, up to a conjecture on multiple zeta values, both associators are not rational. 
\end{abstract}

\maketitle
\tableofcontents
\section{Introduction}
A Drinfeld associator is a formal power series $\Phi(x,y)$ in two non-commuting variables $x$ and $y$ satisfying certain equations. For instance, the easiest equation is
\begin{equation}
\Phi(y,x)^{-1}=\Phi(x,y).
\end{equation}
The set of Drinfeld associators is known to be an infinite dimensional pro-algebraic variety. This means in particular, that there is a wealth of solutions to the aforementioned equations. However, despite this abundance and being studied intensively, Drinfeld associators remain somewhat mysterious objects. In fact, explicit constructions exist only for three associators, the Knizhnik-Zamolodchikov associator $\PhiKZ$, the anti-Knizhnik-Zamolodchikov associator $\PhiAKZ$ and the Alekseev-Torossian associator $\PhiAT$.

A crucial tool to generate new solutions and to relate existing ones to one another is the Grothendieck-Teichm\"{u}ller group $\GRT_1$. It acts freely and transitively on the set of Drinfeld associators. This means that there exists a unique element $\psi\in \GRT_1$ sending the Knizhnik-Zamolodchikov associator $\PhiKZ$ to the anti-Knizhnik-Zamolodchikov associator $\PhiAKZ$ via its action, and a unique element $\tilde{\psi}$ mapping $\PhiKZ$ to the Alekseev-Torossian associator $\PhiAT$. The group $\GRT_1$ is pro-unipotent, which means that it has a (pro-nilpotent) Lie algebra $\grt_1$ and that there exists a bijective exponential map $\exp:\grt_1\rightarrow \GRT_1$ between them. There is thus a unique element $g\in \grt_1$ satisfying $\psi=\exp(g)$. It had been conjectured by P. Etingof that
\begin{equation}
\tilde{\psi}^2=\psi.
\end{equation}
In \cite{Rossi2014}, C. Rossi and T. Willwacher settled this conjecture in the negative. In particular, this implies that by acting with the "square root" of $\psi$, $\psiOH=\exp\left(\frac{g}{2}\right)$, on $\PhiKZ$, one obtains a further associator $\PhiOH$. The relations between the mentioned Drinfeld associators is explained schematically below.
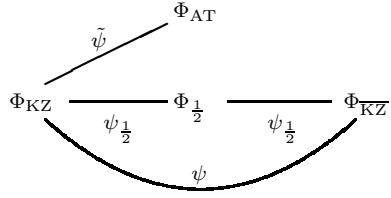
\begin{center}
\setlength{\unitlength}{0.8cm}
\begin{picture}(6,5)
\thicklines
\put(0.9,2.9){\line(2,1){2.}}
\put(3.9,2.6){\line(1,0){1.7}}
\put(1.3,2.6){\line(1,0){1.6}}
\put(0.3,2.5){$\PhiKZ$}
\put(5.8,2.5){$\PhiAKZ$}
\put(3,2.5){$\PhiOH$}
\put(3,4){$\PhiAT$}
\put(1.85,2.15){$\psiOH$}
\put(4.55,2.15){$\psiOH$}
\put(3.3,1.3){$\psi$}
\put(1.65,3.5){$\tilde{\psi}$}
\qbezier(0.9, 2.3)(3.45,0)(6,2.3)
\end{picture}
\end{center}
An interesting part of the theory on Drinfeld associator is the study of the coefficients of these power series, as for instance, they are closely related to multiple zeta values, i.e. the numbers 
\begin{equation}
\zeta(n_1,\dots,n_k)=\sum\limits_{j_1>\dots>j_k\geq 1}{\frac{1}{j_1^{n_1}\cdots j_k^{n_k}}\text{, with } n_1\geq 2, n_i\geq 1 \text{ for all }i\in \{2,\dots ,k\}}.
\end{equation}
In fact, all multiple zeta values appear within the coefficients of the Knizhnik-Zamolodchikov associator $\PhiKZ$. Multiple zeta values were already studied by L. Euler and occur in several branches of mathematics, but are still far from being fully understood. For example, there is an important conjecture on the algebraic relations over $\mathbb{Q}$ satisfied by (regularized) multiple zeta values. These are called shuffle and stuffle relations. The conjecture states that these generate all algebraic relations over $\mathbb{Q}$ among (regularized) multiple zeta values. Our main result (to be stated below) is true up to this conjecture.

One major problem on Drinfeld associators which remained unanswered for a long time is the following.
\begin{question}
Are there associators whose coefficients are all rational?
\end{question}
It was V. Drinfeld who proved that such rational associators must exist \cite{Drinfeld1991}. However, no such element has been found so far. To the author's knowledge, the question whether the Alekseev-Torossian associator $\PhiAT$ or the associator $\PhiOH$ are rational is still unanswered. They are therefore considered potential candidates for being rational associators. It is the aim of this work to tackle this problem. Our main result is the following theorem.
\begin{theorem*}
If the conjecture on multiple zeta values stated above is true, $\PhiAT$ and $\PhiOH$ are not rational associators.
\end{theorem*}
To prove the theorem, we compute the coefficients of $\PhiAT$ and $\PhiOH$ for terms containing up to two $y$'s and arbitrary many $x$'s. We find that (if the conjecture is true) the coefficient of the term $x^2yx^4y$ is irrational for both associators. The computations are based on the comparison of the coefficients appearing in the relations between $\PhiKZ$, $\PhiAKZ$, $\PhiOH$ and $\PhiAT$ induced by the action of the Grothendieck-Teichm\"{u}ller group $\GRT_1$. This endeavour is in fact merely an entertaining combinatorial exercise.

\subsection*{Notation and conventions}
We work over a field $\mathbb{K}$ of characteristic $0$. The \emph{completed free Lie algebra in two generators} will be denoted by 
\begin{equation*}
\cFLiexy:=\prod_{n\geq 1}{\mathbb{F}_{\Lie}(x,y)_n}
\end{equation*}
where $\mathbb{F}_{\Lie}(x,y)$ is the free algebra in two generators and $\mathbb{F}_{\Lie}(x,y)_n$ is spanned by Lie words with $n-1$ brackets. Its topological universal enveloping algebra is $\mathbb{K}\left\langle\left\langle x,y \right\rangle\right\rangle$, i.e. formal power series in the non-commuting variables $x$ and $y$. On $\mathbb{K}\left\langle\left\langle x,y \right\rangle\right\rangle$, the coproduct $\Delta: \mathbb{K}\left\langle \left\langle x,y\right\rangle\right\rangle\rightarrow \mathbb{K}\left\langle \left\langle x,y\right\rangle\right\rangle\hat{\otimes}\mathbb{K}\left\langle \left\langle x,y\right\rangle\right\rangle$ is such that $x$ and $y$ are primitive, that is $\Delta(x)=x\cotimes 1 + 1 \cotimes x$ and $\Delta(y)=y\cotimes 1+ 1 \cotimes y$. Here, $\cotimes$ denotes the completed tensor product.

\subsection*{Acknowledgements} 
I am highly indebted and very grateful to Prof. Thomas Willwacher for his support and advice, and for introducing me to this topic in the context of my Master's thesis.

\section{Drinfeld associators and the Grothendieck-Teichm\"{u}ller group}
The following notions are developed in V. Drinfeld's paper \cite{Drinfeld1991}.

\begin{definition}
The \emph{Drinfeld-Kohno Lie algebra} $\frakt_n$ is the Lie algebra with generators $t_{ij}$, $1\leq i,j\leq n$, $i\neq j$, satisfying the relations
\begin{equation}
[t_{ij},t_{kl}]=0 \text{ for i,j,k,l pairwise distinct,}
\end{equation}
\begin{equation}
[t_{ij},t_{ik}+t_{kj}]=0 \text{ for i,j,k pairwise distinct.}
\end{equation}
\end{definition}

\begin{definition}\label{def: DAss}
A \emph{Drinfeld associator} is a pair $(\mu,\Phi)\in \mathbb{K}^{\times}\times \mathbb{K}\left\langle \left\langle x,y \right\rangle\right\rangle$ such that $\Phi$ is group-like (i.e. $\Delta(\Phi)=\Phi\hat{\otimes}\Phi$) and satisfies the following equations
\begin{align}
\label{eq:DAss1}\Phi(x,y) &= \Phi(y,x)^{-1} \\
	\label{eq:DAss2} 1 &= e^{\frac{\mu}{2}z}\Phi(x,y)e^{\frac{\mu}{2}x}\Phi(y,z)e^{\frac{\mu}{2}y}\Phi(z,x) \\
	\label{eq:DAss3}\Phi(t_{12},t_{23}+t_{24})\Phi(t_{13}+t_{23},t_{34}) &= \Phi(t_{23},t_{34})\Phi(t_{12}+t_{13},t_{24}+t_{34})\Phi(t_{12},t_{23}).
\end{align}
We ask that $x+y+z=0$ in \eqref{eq:DAss2} and that the last equation takes values in the universal enveloping algebra of the Drinfeld-Khono Lie algebra $\frakt_4$. We denote the \emph{set} of Drinfeld associators by $\DAss$.
\end{definition}

\begin{definition}
The \emph{Grothendieck-Teichm\"{u}ller group} $\GRT_1$  is the pro-unipotent group whose elements are solutions $\Phi$ of the equations from Definition \ref{def: DAss} for $\mu=0$. The group operation is given by
\begin{equation}
(\Phi\cdot\Phi')(x,y)=\Phi(x,y)\Phi'(x,\Phi(x,y)^{-1}y\Phi(x,y)),
\end{equation}
for  $\Phi(x,y),\Phi'(x,y)$ in $\GRT_1$ and where the product on the right is the usual product in $\mathbb{K}\left\langle \left\langle x,y \right\rangle\right\rangle$.
\end{definition}
\begin{remark}
The Grothendieck-Teichm\"{u}ller group $\GRT_1$ acts freely and transitively on the set of Drinfeld associators $\DAss$. The action is via
\begin{align*}
\GRT_1\times \DAss &\rightarrow \DAss\\
(\Psi,(\mu,\Phi))&\mapsto (\mu,\Psi\cdot\Phi)
\end{align*}
where the multiplication in the second component is defined in the same manner as the group operation on $\GRT_1$.
\end{remark}

\begin{definition}
We define the \emph{Grothendieck-Teichm\"{u}ller Lie algebra} $(\grt_1,\{\cdot,\cdot\})$ to be given by series $\psi\in \hat{\mathbb{F}}_{\Lie}(x,y)$ satisfying
\begin{align}
\psi(x,y)+\psi(y,x)&=0\\
\psi(x,y)+\psi(y,z)+\psi(z,x)&=0\\
\psi(t_{12},t_{23})+\psi(t_{12}+t_{13},t_{24}+t_{34})+\psi(t_{23},t_{34})&\\
\mathop{-}\psi(t_{12},t_{23}+t_{24})-\psi(t_{13}+t_{23},t_{34})&=0, \nonumber
\end{align}
where $x+y+z=0$ and the last equation takes place in $\frakt_4$. The Lie bracket on $\grt_1$ is given by
\begin{equation}\label{eq:Ihara}
\left\{\psi,\psi'\right\}(x,y)=\left[\psi(x,y),\psi'(x,y)\right]+D_{\psi}\psi'(x,y)-D_{\psi'}\psi(x,y)
\end{equation}
where $D_{\psi}$ is the derivation of the free Lie algebra sending $x$ to $x$ and $y$ to $\left[y,\psi\right]$. The bracket $\{\cdot,\cdot\}$ is sometimes referred to as \emph{Ihara bracket}. 
\end{definition}
The Grothendieck-Teichm\"{u}ller Lie algebra $\grt_1$ is pro-nilpotent. The group $\GRT_1$ can be viewed as the exponential group of $\grt_1$, i.e. $\GRT_1=\exp(\grt_1)$. There is also an action of $\grt_1$ on $\DAss$. It is given by
\begin{align*}
\grt_1\times\DAss&\rightarrow \DAss\\
(\gamma(x,y),(\mu,\Phi))&\mapsto (\mu,\gamma(x,y)\Phi(x,y)+[y,\gamma]\partial_y\Phi(x,y))
\end{align*}
where the first product is the usual product in $\mathbb{K}\left\langle \left\langle x,y \right\rangle\right\rangle$ and $[y,\gamma]\partial_y$ acts as a derivation on $\mathbb{K}\left\langle \left\langle x,y \right\rangle\right\rangle$ sending $x$ to $0$ and $y$ to $[y,\gamma]$.
\begin{example}
Consider the term $x^2yxyx^3$. We have that,
\begin{equation*}
[y,\gamma]\partial_{y}x^2yxyx^3=x^2\left[y,\gamma\right]xyx^3+x^2yx\left[y,\gamma\right]x^3.
\end{equation*}
In the above, this procedure is done for every term of $\Phi$ in a linear way.
\end{example}

\section{Multiple zeta values}
As we will see in the next section, Drinfeld associators are closely related to multiple zeta values which we introduce here. These were already studied by L. Euler. For instance, they appear in \cite{Euler1775}, a text from 1775. In \cite{Euler1735}, he computes the values $\zeta(2n)$ for $n\in\{1,2,3,4,5,6\}$ using a method which works for general $n$.

\begin{definition}
\textit{Multiple zeta values} are given by expressions of the form
\begin{equation}
\zeta(n_1,\dots,n_k):=\sum_{j_1>j_2>\cdots>j_k\geq 1}{\frac{1}{j_{1}^{n_{1}}j_{2}^{n_{2}}\cdots j_k^{n_k}}}
\end{equation}
which are defined as long as $n_1\geq 2$, $n_i\geq 1$ for $i=2,\dots,k$.
\end{definition}

\begin{definition}
Let $B$ be the vector space of formal linear combinations of words $w$ in $x$ and $y$. We say that a word $w$ is \textit{admissible in }B if it is empty or starts with an $x$ and ends in a $y$.
\end{definition}

Every admissible word may be thus written as $x^{n_1-1}yx^{n_2-1}y\cdots x^{n_k-1}y$, with $n_1\geq 2$ and $n_2,\dots,n_k\geq 1$. Let $\Badm$ denote the vector space of formal linear combinations of admissible words. We define a linear map
\begin{equation*}
\zeta:\Badm\rightarrow \mathbb{R}
\end{equation*}
by setting $\zeta(w)=\zeta(n_1,\dots,n_k)$ on words and extending linearly. Moreover, we set $\zeta(\emptyset):=1$.

\begin{definition}
We define the \textit{shuffle product} $\ast$ on $B$ recursively by,
\begin{equation}
\alpha w \ast \alpha' w':=\alpha(w\ast\alpha' w')+\alpha'(\alpha w \ast w'),
\end{equation}
where $w,w'$ are elements of $B$ and $\alpha,\alpha'$ denote the first letter of the words $\alpha w$, $\alpha' w'$, respectively. Furthermore, we set $w\ast \emptyset=\emptyset\ast w=w$ for all $w\in B$.
\end{definition}
The shuffle product of two words corresponds to the sum of all ways of interlacing them. It is associative and commutative. The map $\zeta:(\Badm,\ast)\rightarrow \mathbb{R}$ is a homomorphism of commutative algebras . In fact, this map extends uniquely to a morphism of algebras $\zeta:(B,\ast)\rightarrow \mathbb{R}$ such that $\zeta(x)=\zeta(y)=0$. This regularizing process is maybe best explained by looking at a simple example. The multiple zeta values which may be assigned also to non-admissible words in this way are called \emph{shuffle regularized multiple zeta values}.
\begin{example}
The word $w=x^2yx^2$ ends with an $x$ and is thus not admissible. We may rewrite it as,
\begin{align*}
x^2yx^2 &= x^2yx\ast x - x^2yx^2-x^3yx-x^3yx-x^3yx\\
\Leftrightarrow x^2yx^2 &= \frac{1}{2}(x^2yx\ast x -3x^3yx)=\frac{1}{2}(x^2yx\ast x-3(x^3y\ast x - 4x^4y))\\
&=\frac{1}{2}(x^2yx\ast x-3x^3y\ast x +12x^4y)
\end{align*}
Therefore, 
\begin{align*}
\zeta(x^2yx^2)&=\frac{1}{2}(\zeta(x^2yx\ast x)-3\zeta(x^3y\ast x) +12\zeta(x^4y))\\
&=\frac{1}{2}(\zeta(x^2yx)\zeta(x)-3\zeta(x^3y)\zeta(x) + 12\zeta(x^4y))\\
&=0+0+6\zeta(x^4y)=6\zeta(x^{5-1}y)=6\zeta(5).
\end{align*}
\end{example}
This example may be generalized to the following statement.
\begin{lemma}\label{lemma:zeta}
Let $w=x^ayx^b$. Then $\zeta(w)=(-1)^b\cfrac{(a+b)!}{a!b!}\zeta(a+b+1)$.
\end{lemma}
\begin{proof}
We will have to consider the shuffle product of $x^ayx^{b-1}$ and $x$. For this we need to look at all ways of interlacing the two words and to illustrate this process, we mark the word $x$ by $\underline{x}$. Write
\begin{align*}
\zeta(x^ayx^b)=\zeta(x^ayx^{b-1})\zeta(x)&-\zeta(x^ayx^{b-2}\underline{x}x)\\
&-\zeta(x^ayx^{b-3}{\underline{x}}x^2)\\
&-\dots\\
&-\zeta(x^ay{{\underline{x}}} x^{b-1})\\
&-\zeta(x^a{\underline{x}} yx^{b-1})\\
&-\zeta(x^{a-1}{\underline{x}}xyx^{b-1})\\
&-\dots\\
&-\zeta({\underline{x}}x^ayx^{b-1})\\
=0-(b-1)\zeta(x^ayx^b)&-(a+1)\zeta(x^{a+1}yx^{b-1})\\
\end{align*}
Hence,
\begin{align*}
\zeta(x^ayx^b)&=-\frac{a+1}{b}\zeta(x^{a+1}yx^{b-1})\\
&=\frac{a+1}{b}\cdot\frac{a+2}{b-1}\zeta(x^{a+2}yx^{b-2})\\
&=\dots\\
&=(-1)^b\frac{(a+1)\cdots(a+b)}{b!}\zeta(x^{a+b}y)=(-1)^b\cfrac{(a+b)!}{a!b!}\zeta(a+b+1)
\end{align*}
\end{proof}

The relations between multiple zeta values which are obtained using the shuffle product are called \textit{shuffle relations}. To illustrate this notion, we look at the following example.
\begin{example}
We may write $\zeta(2)$ as $\zeta(xy)$. This way,
\begin{align*}
\zeta(2)^2&=\zeta(xy)\zeta(xy)=\zeta(xy\ast \underline{xy})\\
&=\zeta(xy\underline{xy})+\zeta(x\underline{x}y\underline{y})+\zeta(x\underline{xy}y)+\zeta(\underline{x}xy\underline{y})+\zeta(\underline{x}x\underline{y}y)+\zeta(\underline{xy}xy)\\
&=4\zeta(x^2y^2)+2\zeta(xyxy)=4\zeta(3,1)+2\zeta(2,2).
\end{align*}
Again, we have marked the letters of the word $xy$, so that we may keep track of them as we shuffle.
\end{example}

Next we consider the vector space $A$ of formal linear combinations of words $u=n_1\cdots n_k$ with letters $n_i$ from the alphabet $\mathbb{N}$. Again there is a notion of \textit{admissible words}. We say that $u=n_1\cdots n_k$ is \textit{admissible in }A if $n_1\geq 2$ (or $k=0$). Denote by $\Aadm$ the vector space of formal linear combinations of admissible words. In analogy to the previous construction, we define a linear map 
\begin{equation*}
\zeta:\Aadm\rightarrow \mathbb{R}
\end{equation*}
by setting $\zeta(u)=\zeta(n_1,\dots,n_k)$ on words $u=n_1\cdots n_k$ and extending linearly. Moreover, $\zeta(\emptyset):=1$.
\begin{definition}
On $A$ we may define recursively the \textit{stuffle product} $\sha$, by
\begin{equation}
nw\sha n'w':=n(w\sha n'w')+(n+n')(w\sha w')+n'(nw\sha w'),
\end{equation}
where $n,n'\in \mathbb{N}$ and $w,w'$ are words. Furthermore, set $u\sha \emptyset=\emptyset\sha u=u$.
\end{definition}
\begin{example}
Simply by applying the definition,
\begin{align*}
23\sha 5&=2(3\sha 5)+(2+5)(3\sha \emptyset)+5(23\sha \emptyset)\\
&=2(3(\emptyset\sha 5)+8(\emptyset\sha\emptyset)+5(3\sha\emptyset))+73+523\\
&=235+28+253+73+523.
\end{align*}
\end{example}
The stuffle product is associative and commutative. Moreover, the map $\zeta:(\Aadm,\sha)\rightarrow \mathbb{R}$ defines an algebra homomorphism, i.e. for admissible words $u,u'$, $\zeta(u\sha u')=\zeta(u)\zeta(u')$. It may be extended uniquely to a map of algebras $\zeta:(A,\sha)\rightarrow \mathbb{R}$ such that $\zeta(1)=0$.
As before, the regularizing process is probably best explained via a simple example. The multiple zeta values which may be assigned also to non-admissible words in this way are called \textit{stuffle regularized multiple zeta values}.
\begin{example}
The word $w=123$ begins with a one and is therefore not admissible. We note that
\begin{align*}
1\sha 23&=1(\emptyset\sha 23)+(1+2)(\emptyset\sha 3)+2(1\sha 3)\\
&=123+33+2\left(1(\emptyset\sha 3)+(1+3)(\emptyset\sha\emptyset)+3(1\sha\emptyset)\right)\\
&=123+33+213+24+231.
\end{align*}
Therefore, by applying the map $\zeta$
\begin{align*}
\zeta(1,2,3)=\zeta(123)&=\zeta(1\sha 23)-\zeta(33)-\zeta(213)-\zeta(24)-\zeta(231)\\
&=\zeta(1)\zeta(23)-\zeta(33)-\zeta(213)-\zeta(24)-\zeta(231)\\
&=-\zeta(3,3)-\zeta(2,1,3)-\zeta(2,4)-\zeta(2,3,1).
\end{align*}
\end{example}
The relations between multiple zeta values which are obtained using the stuffle product are called \textit{stuffle relations}.
\begin{example}
We will make use of the following stuffle relation. Let $r,s\in \mathbb{N}$.
\begin{align*}
\zeta(r)\zeta(s)&=\zeta(r\sha s)=\zeta(r(\emptyset\sha s)+(r+s)(\emptyset\sha\emptyset)+s(r\sha\emptyset))\\
&=\zeta(rs)+\zeta(r+s)+\zeta(sr)=\zeta(r,s)+\zeta(r+s)+\zeta(s,r)
\end{align*}
\end{example}

It is conjectured (though apparently with no solution in sight) that the shuffle and stuffle relations, together with a relation which relates the shuffle and stuffle relations to each other \cite{Furusho}, generate \textit{all} algebraic relations satisfied by the regularized shuffle and the regularized stuffle multiple zeta values over $\mathbb{Q}$. If this were the case, then the conjecture below would be true.
\begin{definition}
Consider $\zeta(n_1,\dots,n_k)$ with $n_1\geq 2,n_i\geq 1$ for $i=2,\dots,k$. We call $N=n_1+\dots+n_k$ the \textit{weight} of the multiple zeta value $\zeta(n_1,\dots,n_k)$.
\end{definition}

A reference for the following conjecture is F. Brown's paper on the decomposition of multiple zeta values \cite{Brown2012}. 
\begin{conjecture}\label{conjecture}
Let $\mathcal{Z}_N$ denote the $\mathbb{Q}$-vector space spanned by the set of multiple zeta values $\zeta(n_1,\dots,n_k)$ with $n_1\geq 2,n_i\geq 1$ for $i=2,\dots,k$ of weight $N$. Up to weight 8, the spaces $\mathcal{Z}_N$ are \textit{conjectured} to have the following bases over $\mathbb{Q}$:\newline
\begin{center}
\begin{tabular}{|c|c|c|c|c|c|c|c|c|}
\hline
Weight $N$ & 1 & 2 & 3 & 4 & 5 & 6 & 7 & 8\\\hline
$\mathbb{Q}$-Basis for $\mathcal{Z}_N$ & $\emptyset$ & $\zeta(2)$ & $\zeta(3)$ & $\zeta(2)^2$ & $\zeta(5)$ & $\zeta(3)^2$ & $\zeta(7)$ & $\zeta(3,5)$ \\
  & & & & & $\zeta(3)\zeta(2)$ & $\zeta(2)^3$ & $\zeta(5)\zeta(2)$ & $\zeta(3)\zeta(5)$\\
	& & & & & & & $\zeta(3)\zeta(2)^2$ & $\zeta(3)^2\zeta(2)$ \\
	& & & & & & & & $\zeta(2)^4$\\\hline
\end{tabular}
\end{center}
\end{conjecture}

\begin{example}
This would imply that, for instance, there must be a relation between the multiple zeta values $\zeta(3)$ and $\zeta(2,1)$ which are both of weight 3. Indeed, the stuffle relation gives,
\begin{equation*}
0=\zeta(2)\zeta(1)=\zeta(2,1)+\zeta(3)+\zeta(1,2)
\end{equation*}
as $\zeta(1)=\zeta(y)=0$. On the other hand, shuffle relations yield,
\begin{equation*}
\zeta(1,2)=-2\zeta(2,1)
\end{equation*}
which eventually implies
\begin{equation*}
\zeta(3)=\zeta(2,1).
\end{equation*}
\end{example}

Next, we consider a somewhat more elaborate example.
\begin{exampleprop}
In the conjectural $\mathbb{Q}$-basis $\left\{\zeta(2)^3,\zeta(3)^2\right\}$ for the space of multiple zeta values of weight 6, $\zeta(4,2)$ may be expressed as
\begin{equation*}
\zeta(4,2)=\zeta(3)^2-\frac{32}{105}\zeta(2)^3.
\end{equation*}
\end{exampleprop}
\begin{proof}
Using the shuffle relations, we find
\begin{align*}
\zeta(4)\zeta(2)=&\zeta(x^3y\ast xy)\\
=&4\zeta(x^3yxy)+8\zeta(x^4y^2)+2\zeta(x^2yx^2y)+\zeta(xyx^3y)\\
=&4\zeta(4,2)+8\zeta(5,1)+2\zeta(3,3)+\zeta(2,4)
\end{align*}
The stuffle relation
\begin{equation*}
\zeta(r)\zeta(s)=\zeta(r,s)+\zeta(r+s)+\zeta(s,r)
\end{equation*}
gives
\begin{equation*}
\zeta(2,4)=\zeta(4)\zeta(2)-\zeta(6)-\zeta(4,2)
\end{equation*}
and
\begin{equation*}
2\zeta(3,3)=\zeta(3)^2-\zeta(6).
\end{equation*}
Moreover, we have the following identity (apparently due to Euler\cite{Broadhurst}). For $a>1$,
\begin{equation*}
\zeta(a,1)=\frac{5}{2}\zeta(a+1)-\frac{1}{2}\sum\limits_{b=2}^{a}{\zeta(a+1-b)\zeta(b)}.
\end{equation*}
This allows us to write
\begin{align*}
\zeta(5,1)&=\frac{5}{2}\zeta(6)-\frac{1}{2}\left(\zeta(4)\zeta(2)+\zeta(3)^2+\zeta(2)\zeta(4)\right)\\
&=\frac{5}{2}\zeta(6)-\frac{1}{2}\zeta(3)^2-\zeta(2)\zeta(4).
\end{align*}
Therefore,
\begin{align*}
\zeta(4)\zeta(2)=&4\zeta(4,2)+8\zeta(5,1)+2\zeta(3,3)+\zeta(2,4)\\
=&4\zeta(4,2)+8\left(\frac{5}{2}\zeta(6)-\frac{1}{2}\zeta(3)^2-\zeta(2)\zeta(4)\right)+\zeta(3)^2-\zeta(6)\\
&+\zeta(4)\zeta(2)-\zeta(6)-\zeta(4,2)\\
=&3\zeta(4,2)+18\zeta(6)-3\zeta(3)^2-7\zeta(2)\zeta(4)\\
\Leftrightarrow 3\zeta(4,2)=&3\zeta(3)^2+8\zeta(2)\zeta(4)-18\zeta(6).
\end{align*}
Now, using $\zeta(2)=\frac{\pi^2}{6}$, $\zeta(4)=\frac{\pi^4}{90}$ and $\zeta(6)=\frac{\pi^6}{945}$\cite{Euler1735}, we have
\begin{equation*}
4\zeta(2)\zeta(4)=7\zeta(6)
\end{equation*}
and
\begin{equation*}
\zeta(6)=\frac{8}{35}\zeta(2)^3.
\end{equation*}
This simplifies the equation above further, namely,
\begin{equation*}
3\zeta(4,2)=3\zeta(3)^2-4\zeta(6)=3\zeta(3)^2-\frac{32}{35}\zeta(2)^3.
\end{equation*}
Thus, eventually,
\begin{equation*}
\zeta(4,2)=\zeta(3)^2-\frac{32}{105}\zeta(2)^3.
\end{equation*}
\end{proof}

\section{Explicit Drinfeld associators}
We are now in a position to give an explicit description of two Drinfeld associators.
\subsection{The Knizhnik-Zamolodchikov associators}
The Drinfeld associators defined below where first defined by V. Drinfeld \cite{Drinfeld1991}. Further explanations can be found in \cite{Le1996}.
\begin{definition}\cite{Drinfeld1991}
The \textit{Knizhnik-Zamolodchikov (KZ) associator} is given by the following formula. 
\begin{equation}
\PhiKZ(x,y):=1+\sum_{w\in B,\left|w\right|\geq 2}{\frac{(-1)^{n_w}}{(2\pi i)^{\left|w\right|}}\zeta(w)w}
\end{equation}
Here, $\left|w\right|$ denotes the length of the word $w$, $n_w$ the number of $y$'s in $w$ and $\zeta(w)$ the regularized multiple zeta value associated to $w$. 
The \textit{anti-Knizhnik-Zamolodchikov associator} is defined as
\begin{equation}
\PhiAKZ(x,y):=\PhiKZ(-x,-y).
\end{equation}
\end{definition}
\begin{remark}
\cite{Drinfeld1991} Both the KZ associator and the anti-KZ associator are Drinfeld associators. They satisfy equations \eqref{eq:DAss1} to \eqref{eq:DAss3} for the parameter $\mu=1$.
\end{remark}

\begin{remark}
Note that indeed $\PhiAKZ\neq\PhiKZ$ as the coefficients of words of odd length will have opposite signs. For instance, 
\begin{equation*}
0 \neq u_{x^2y}^{\text{KZ}}=-\frac{\zeta(3)}{(2\pi i)^3}\neq +\frac{\zeta(3)}{(2\pi i)^3}=u_{x^2y}^{\overline{\text{KZ}}}.
\end{equation*}
\end{remark}

\begin{remark}
Since $\zeta(x)=\zeta(y)=0$, also $\zeta(x^n)=0$ for all $n$, via
\begin{equation*}
\zeta(x^n)=\zeta(x^{n-1})\zeta(x)-(n-1)\zeta(x^{n})\Leftrightarrow n\zeta(x^n)=0
\end{equation*}
and similarly for $y$. This implies that in $\PhiKZ$ and $\PhiAKZ$ terms containing only $x$'s or only $y$'s do not appear.
\end{remark}

\subsection{The first main Theorem}

Since we have a free and transitive action of $\GRT_1$ on $\DAss$, there exists a unique element $\psi\in \GRT_1$ such that $\psi\cdot\PhiKZ=\PhiAKZ$. The group $\GRT_1$ is pro-unipotent, and thus there is a unique element $g \in \grt_1$ with $\psi=\exp(g)$. Therefore, it makes sense to consider the "square root" of $\psi$, 
\begin{equation}
\psiOH:=\exp(\frac{g}{2})\in \GRT_1.
\end{equation}
It is the element that satisfies the equation
\begin{equation}
(\psiOH\cdot\psiOH)(x,y)=\psi(x,y),
\end{equation}
where the product on the left is again the product in $\GRT_1$. Let us spell out in detail, why the equation above holds for $\psiOH(x,y):=\exp(\frac{1}{2}g(x,y))$, as one has to be careful with the peculiar products and brackets appearing.
\begin{align*}
(\psiOH\cdot\psiOH)(x,y)&=\exp(\frac{1}{2}g(x,y))\cdot\exp(\frac{1}{2}g(x,y))\\
&=\exp\left(\bch(\frac{g}{2},\frac{g}{2})(x,y)\right)\\
&=\exp(\frac{g(x,y)}{2}+\frac{g(x,y)}{2}+\frac{1}{2}\left\{\frac{g}{2},\frac{g}{2}\right\}(x,y)+\dots)\\
&=\exp(g(x,y))=\psi(x,y)
\end{align*}
The Baker-Campbell-Hausdorff element $\bch(-,-)$ is not taken with respect to the ordinary bracket of Lie series, but the Ihara bracket (see equation \eqref{eq:Ihara}). By acting with this new element $\psiOH\in \GRT_1$ on $\PhiKZ$, we find a Drinfeld associator \cite{Rossi2014}, which we'll denote $\PhiOH$, i.e.
\begin{equation}
\PhiOH(x,y):=(\psiOH\cdot\PhiKZ)(x,y).
\end{equation}

In \cite{Drinfeld1991}, it was shown that there must exist an associator $\Phi(x,y)\in \mathbb{K}\left\langle \left\langle x,y \right\rangle\right\rangle$ having only rational coefficients. The natural (and, to the author's knowledge, also unanswered) question that arises here is thus, whether the associator $\PhiOH$ is in fact such an associator. Unfortunately, this does not seem to be the case. Our first main result is the following theorem.

\begin{theorem}\label{thm:OH}
The coefficient in $\PhiOH$ of the word $w=x^2yx^4y$ is
\begin{equation*}
\cfrac{2\zeta(3,5)-7\zeta(3)\zeta(5)}{512\pi^8}.
\end{equation*}
If the conjectured basis $\left\{\zeta(3,5),\zeta(3)\zeta(5),\zeta(3)^2\zeta(2),\zeta(2)^4\right\}$ from conjecture \ref{conjecture} is in fact a $\mathbb{Q}$-basis for multiple zeta values of weight 8, then this coefficient is irrational.
\end{theorem}

\subsection{Towards a proof of Theorem \ref{thm:OH}}

We start by considering two general group-like formal power series $\Phi,\Phi'\in \mathbb{K}\left\langle \left\langle x,y \right\rangle\right\rangle$ satisfying the equations from definition \ref{def: DAss} for some $\mu\in \mathbb{K}$. We want to understand the first terms of their product (as if they were either two elements of $\GRT_1$ or $\Phi\in \GRT_1$ and $\Phi'\in \DAss$)
\begin{equation}\label{eq:product}
(\Phi\cdot\Phi')(x,y)=\Phi(x,y)\Phi'(x,\Phi(x,y)^{-1}y\Phi(x,y)).
\end{equation}

\begin{lemma}
A group-like element $\Phi(x,y)$ satisfying the equations \eqref{eq:DAss1} to \eqref{eq:DAss3}, will not contain any linear terms in $x$ and $y$. 
\end{lemma}
\begin{proof}
Assume on the contrary that
\begin{equation*}
\Phi(x,y)=1+ax+by+(\dots)
\end{equation*}
where $a,b\neq 0$ and $(\dots)$ denotes higher order terms. The antisymmetry relation implies $a=-b$, via
\begin{align*}
\Phi(x,y)\Phi(y,x)&=1\\
\Leftrightarrow(1+ax+by+(\dots))\dot(1+ay+bx+(\dots))&=1\\
\Leftrightarrow(1+(a+b)(x+y)+(\dots))&=1\\
\Leftrightarrow a+b=0.
\end{align*}
Inserting this into the third equation \eqref{eq:DAss3}, we get to first order,
\begin{align*}
&1+a(t_{12}-t_{23}-t_{24})+a(t_{13}+t_{23}-t_{34})\\
&=1+a(t_{23}-t_{34})+a(t_{12}+t_{13}-t_{24}-t_{34})+a(t_{12}-t_{23})\\
&\Leftrightarrow t_{12}-t_{23}-t_{24}+t_{13}+t_{23}-t_{34}\\
&=t_{23}-t_{34}+t_{12}+t_{13}-t_{24}-t_{34}+t_{12}-t_{23}\\
&\Leftrightarrow 0=t_{12}-t_{34}.
\end{align*}
This is a contradiction, and hence $\Phi$ cannot contain any terms of first order.
\end{proof}

We may therefore represent the formal power series by
\begin{align*}
\Phi(x,y)&=1+\sum_{w\in B,\left|w\right|\geq 2}{a_{w}w}  \\
\text{ and }\Phi'(x,y)&=1+\sum_{w\in B,\left|w\right|\geq 2}{b_{w}w}.
\end{align*}
For our purposes, we may furthermore assume that coefficients of words containing only $x$'s or only $y$'s will equal zero in both $\Phi$ and $\Phi'$.

\begin{definition}
The \textit{degree in }$y$ (in $x$) of a word $w$ in $B$ is the number of $y$'s (of $x$'s) in $w$. We will mostly look at the degree in $y$, and in this case simply refer to it as \textit{the degree}.
\end{definition}

\begin{remark}
The anti-symmetry equation $\Phi(x,y)^{-1}=\Phi(y,x)$, implies that
\begin{align*}
a_{w(x,y)}&=-a_{w(y,x)}\\
b_{w(x,y)}&=-b_{w(y,x)}
\end{align*}
for all $w$ of degree one in either $x$ or $y$.
\end{remark}

We now want to find the coefficients $p_w$ up to degree 2 of the product 
\begin{equation}
(\Phi\cdot\Phi')(x,y):=1+\sum\limits_{w\in B,\left|w\right|\geq 2}{p_{w}w}
\end{equation}
This is done by comparing the coefficients of the words $x^ayx^b$ (for degree 1), $x^ayx^byx^c$ (for degree 2) for $a,b,c\in \mathbb{N}_0$ in the equation
\begin{equation}
\Phi(x,y)\Phi'(x,\Phi(x,y)^{-1}y\Phi(x,y))=1+\sum\limits_{w\in B,\left|w\right|\geq 2}{p_{w}w}.
\end{equation}

\begin{degree1}
Let thus $w=x^ayx^b$ where $a,b\in \mathbb{N}_{0}$. We find
\begin{equation}\label{eq:deg1}
\begin{aligned}
p_w&=a_w+b_w\\
\text{and } p_{w(x,y)}&=-p_{w(y,x)}.
\end{aligned}
\end{equation}
\end{degree1}

\begin{degree2}
Let $w=x^ayx^byx^c$ with $a,b,c\in \mathbb{N}_{0}$. Then
\begin{equation}
\begin{aligned}\label{eq:deg2}
p_{w}=&a_{w}\cdot 1+ 1 \cdot b_{w}\\
&+\sum_{j=0}^{b}{a_{x^ayx^j}b_{x^{b-j}yx^c}}\\
&+\sum_{j=0}^{a}{b_{x^jyx^c} (-a_{x^{a-j}yx^b})}\\
&+\sum_{j=0}^{c}{b_{x^ayx^j} a_{x^byx^{c-j}}}.
\end{aligned}
\end{equation}
\end{degree2}

Using the formulas above we may solve the equation
\begin{equation}
(\psi\cdot\PhiKZ)(x,y)=\PhiAKZ(x,y)
\end{equation}
for the coefficients of $\psi\in \GRT_1$ up to degree 2. For this, let 
\begin{equation}
\psi(x,y):=1+\sum\limits_{w\in B,\left|w\right|\geq2}{c_{w}w}
\end{equation}
and write $u_w$ for the coefficients in $\PhiKZ$, i.e.
\begin{equation}
\PhiKZ(x,y):=1+\sum\limits_{w\in B,\left|w\right|\geq2}{u_{w}w}.
\end{equation}

\begin{degree1}
Let $w=x^ayx^b$ with $a,b\in \mathbb{N}_{0}$. Then from equation \eqref{eq:deg1}, we find,
\begin{equation}
c_{w}+u_{w} =
\left\{
	\begin{array}{ll}
		-u_{w}  & \mbox{if } \left|w\right| \text{ odd} \\
		u_{w} & \mbox{if } \left|w\right| \text{ even}
	\end{array}
\right.
\end{equation}
This implies,
\begin{equation}
c_{w} =
\left\{
	\begin{array}{ll}
		-2u_{w}  & \mbox{if } \left|w\right| \text{ odd} \\
		0 & \mbox{if } \left|w\right| \text{ even.}
	\end{array}
\right.
\end{equation}
Moreover, as for all elements of $\GRT_1$ and $\DAss$, $c_{w(x,y)}=-c_{w(y,x)}$, whenever $w$ is of degree 1.
\end{degree1}

\begin{degree2}
Let $w=x^ayx^byx^c$ with $a,b,c\in \mathbb{N}_{0}$. Equation \eqref{eq:deg2} gives,
\begin{equation}\label{eq:cw2}
\begin{aligned}
c_{w}=&-u_{w}+(-1)^{\left|w\right|}u_{w}-\sum_{j=0}^{b}{c_{x^ayx^j}u_{x^{b-j}yx^c}}\\
&-\sum_{j=0}^{a}{u_{x^jyx^c} (-c_{x^{a-j}yx^b})}-\sum_{j=0}^{c}{u_{x^ayx^j} c_{x^byx^{c-j}}}.
\end{aligned}
\end{equation}
We may simplify the expression for $c_w$ by introducing the following notation. For a word $v=x^pyx^q$ of degree 1, define
\begin{equation*}
\delta_{pq} =
\left\{
	\begin{array}{ll}
		1  & \mbox{if } \left|v\right|=p+q+1 \text{ odd} \\
		0 & \mbox{if } \left|v\right|=p+q+1 \text{ even}
	\end{array}
\right.
\end{equation*}
 Then $c_v=\delta_{pq}\cdot (-2u_v)$. This way we may rewrite equation \eqref{eq:cw2} as,
\begin{equation}\label{eq:cw}
\begin{aligned}
c_{w}=&-u_{w}+(-1)^{\left|w\right|}u_{w}+2\sum_{j=0}^{b}{\delta_{aj} u_{x^ayx^j}u_{x^{b-j}yx^c}}\\
&+2\sum_{j=0}^{a}{\delta_{(a-j)b} u_{x^jyx^c} (-u_{x^{a-j}yx^b})}+2\sum_{j=0}^{c}{\delta_{b(c-j)}u_{x^ayx^j} u_{x^byx^{c-j}}}.
\end{aligned}
\end{equation}
This expression will turn out to be useful later on. Note that in the case where there should be no confusion, we might write $\delta_v$ instead of $\delta_{pq}$.
\end{degree2}


Our next objective will be to find the coefficients up to degree 2 of $\psiOH$. We write $d_w$ for the coefficients of $\psiOH$, i.e.
\begin{equation}
\psiOH(x,y):=1+\sum\limits_{w\in B,\left|w\right|\geq 2}{d_{w}w}.
\end{equation}
The defining equation for $\psiOH$ is
\begin{equation}
(\psiOH\cdot\psiOH)(x,y)=\psi(x,y).
\end{equation}

\begin{degree1}
Let $w=x^ayx^b$ with $a,b \in \mathbb{N}_{0}$. From the above (equation \eqref{eq:deg1}) we get,
\begin{equation}\label{eq:dw1}
\begin{aligned}
d_{w}+d_{w}&=c_{w}\\
\Leftrightarrow d_{w}&=\frac{c_{w}}{2}=\left\{
	\begin{array}{ll}
		\frac{-2u_w}{2}=-u_w  & \mbox{if } \left|w\right| \text{ odd} \\
		0 & \mbox{if } \left|w\right| \text{ even.}
	\end{array}
\right.
\end{aligned}
\end{equation}
Additionally, $d_{w(x,y)}=-d_{w(y,x)}$, for a word $w$ of degree 1.
\end{degree1}

\begin{degree2}
Consider $w=x^ayx^byx^c$ with $a,b,c \in \mathbb{N}_{0}$. Then, by equation \eqref{eq:deg2},
\begin{align*}
d_{w}=&\frac{1}{2}\Biggl(\Biggr.c_{w}-\sum_{j=0}^{b}{d_{x^ayx^j}d_{x^{b-j}yx^c}}\\
&-\sum_{j=0}^{a}{d_{x^jyx^c} (-d_{x^{a-j}yx^b})}-\sum_{j=0}^{c}{d_{x^ayx^j} d_{x^byx^{c-j}}}\Biggl.\Biggr).
\end{align*}
Again, we may rewrite this solely in terms of $c_w$ and $u_v$'s for $v$'s of degree 1. Namely, if $v=x^pyx^q$, then $d_v=-\delta_{pq}\cdot u_v$ (equation \eqref{eq:dw1}). Therefore
\begin{align*}
d_{w}=&\frac{1}{2}\Bigg(c_{w}-\sum_{j=0}^{b}{\delta_{aj}\delta_{(b-j)c}u_{x^ayx^j}u_{x^{b-j}yx^c}}\\
&-\sum_{j=0}^{a}{\delta_{jc}\delta_{(a-j)b} u_{x^jyx^c} (-u_{x^{a-j}yx^b})}-\sum_{j=0}^{c}{\delta_{aj}\delta_{b(c-j)} u_{x^ayx^j} u_{x^byx^{c-j}}}\Bigg).
\end{align*}
Now, to make matters worse, we have to consider the cases where $\left|w\right|$ is odd and where it is even separately. If $\left|w\right|$ is odd, then every $u_v u_v'$ in the sums above, with $v,v'$ of degree 1, will be such that one of $\left|v\right|$ and $\left|v'\right|$ is even and the other is odd. Therefore, products of $\delta$'s appearing in the sums above will give 0, and hence,
\begin{equation}\label{eq:dw2odd}
d_w=\frac{c_w}{2}.
\end{equation}

If, on the other hand $\left|w\right|$ is even, all $u_v u_v'$ in the sum will correspond to words $v,v'$ with lengths of equal parity. In this case, the products $\delta_v\delta_{v'}=\delta_v=\delta_{v'}$, and we may rewrite $d_w$ as,

\begin{align*}
d_{w}=&\frac{1}{2}\Bigg(c_{w}-\sum_{j=0}^{b}{\delta_{aj}u_{x^ayx^j}u_{x^{b-j}yx^c}}\\
&-\sum_{j=0}^{a}{\delta_{(a-j)b} u_{x^jyx^c} (-u_{x^{a-j}yx^b})}-\sum_{j=0}^{c}{\delta_{b(c-j)} u_{x^ayx^j} u_{x^byx^{c-j}}}\Bigg).
\end{align*}
Next, we replace $c_w$ with the expression found in equation \eqref{eq:cw} to find,
\begin{align*}
d_w=&\frac{1}{2}\Bigg(-u_{w}+(-1)^{\left|w\right|}u_{w}+2\underbrace{\sum_{j=0}^{b}{\delta_{aj} u_{x^ayx^j}u_{x^{b-j}yx^c}}}_{(\ast)}\\
&+2\underbrace{\sum_{j=0}^{a}{\delta_{(a-j)b} u_{x^jyx^c} (-u_{x^{a-j}yx^b})}}_{(\ast\ast)}+2\underbrace{\sum_{j=0}^{c}{\delta_{b(c-j)}u_{x^ayx^j} u_{x^byx^{c-j}}}}_{(\ast\ast\ast)}\\
&-\underbrace{\sum_{j=0}^{b}{\delta_{aj}u_{x^ayx^j}u_{x^{b-j}yx^c}}}_{(\ast)}-\underbrace{\sum_{j=0}^{a}{\delta_{(a-j)b} u_{x^jyx^c} (-u_{x^{a-j}yx^b})}}_{(\ast\ast)}-\underbrace{\sum_{j=0}^{c}{\delta_{b(c-j)} u_{x^ayx^j} u_{x^byx^{c-j}}}}_{(\ast\ast\ast)}\Bigg).
\end{align*}
We have marked equal terms to clarify why the expression simplifies to,
\begin{equation}\label{eq:dw}
\begin{aligned}
d_w=&\frac{1}{2}\Bigg(-u_{w}+(-1)^{\left|w\right|}u_{w}+\sum_{j=0}^{b}{\delta_{aj} u_{x^ayx^j}u_{x^{b-j}yx^c}}\\
&+\sum_{j=0}^{a}{\delta_{(a-j)b} u_{x^jyx^c} (-u_{x^{a-j}yx^b})}+\sum_{j=0}^{c}{\delta_{b(c-j)}u_{x^ayx^j} u_{x^byx^{c-j}}}\Bigg).
\end{aligned}
\end{equation}
\end{degree2}


The coefficients of the associator $\PhiOH$, denoted $f_w$, are found by looking at the equation
\begin{equation}
(\psiOH\cdot\PhiKZ)(x,y)=\PhiOH(x,y).
\end{equation}

\begin{degree1}
For $w=x^ayx^b$, $a,b\in \mathbb{N}_{0}$, we have by equation \eqref{eq:deg1},
\begin{equation*}
f_{w}=d_{w}+u_{w}.
\end{equation*}
Together with (coming from equation \eqref{eq:dw1})
\begin{equation*}
d_w=\frac{c_w}{2}=\left\{
	\begin{array}{ll}
		-u_w  & \mbox{if } \left|w\right| \text{ odd} \\
		0 & \mbox{if } \left|w\right| \text{ even,}
	\end{array}
\right.
\end{equation*}
this implies
\begin{equation*}
f_w=\left\{
	\begin{array}{ll}
		-u_{w}+u_w=0  & \mbox{if } \left|w\right| \text{ odd} \\
		u_w & \mbox{if } \left|w\right| \text{ even.}
	\end{array}
\right.
\end{equation*}

Hence, $\PhiOH$ is in fact a new Drinfeld associator, i.e. it is not equal to $\PhiKZ$ or $\PhiAKZ$.

\begin{remark}
For $w=x^pyx^q$ with $\left|w\right|=p+q+1=2n$ even, the coefficient $u_w$ has a particularly nice form.
\begin{align*}
u_w&=\frac{-1}{(2\pi i)^{2n}}\zeta(w)=\frac{-1}{(2\pi i)^{2n}}(-1)^q\frac{(p+q)!}{p!q!}\zeta(p+q+1)\\
&=\frac{(-1)^{q+1}(p+q)!}{(2\pi i)^{2n}p!q!}\zeta(2n).
\end{align*}
Note that we used Lemma \ref{lemma:zeta} to regularize $\zeta(w)$. It is known, that $\zeta(2n)=(-1)^{n-1}\frac{(2\pi)^{2n}}{2(2n)!}\text{B}_{2n}$, where $\text{B}_{2n}\in \mathbb{Q}$ is the $2n$-th Bernoulli-number \cite{Broadhurst}. More important than the explicit form of $\zeta(2n)$ is the fact that $\zeta(2n)=q\cdot \pi^{2n}$, for some $q\in \mathbb{Q}$. This implies that $u_w$ is a rational number for all words $w$ of degree 1.
\end{remark}

\end{degree1}

\begin{degree2}
Let $w=x^ayx^byx^c$ with $a,b,c\in \mathbb{N}_{0}$. Then following equation \eqref{eq:deg2},
\begin{align*}
f_{w}=&d_{w}\cdot 1+ 1 \cdot u_{w}\\
&+\sum_{j=0}^{b}{d_{x^ayx^j}u_{x^{b-j}yx^c}}\\
&+\sum_{j=0}^{a}{u_{x^jyx^c} (-d_{x^{a-j}yx^b})}\\
&+\sum_{j=0}^{c}{u_{x^ayx^j} d_{x^byx^{c-j}}}.
\end{align*}

\begin{case1}
Assume that $\left|w\right|$ is odd. Then $d_w=\cfrac{c_w}{2}$ (equation \eqref{eq:dw2odd}) and
\begin{align*}
f_{w}=&\frac{c_w}{2}+ u_{w}\\
&-\sum_{j=0}^{b}{\delta_{aj} u_{x^ayx^j}u_{x^{b-j}yx^c}}\\
&-\sum_{j=0}^{a}{\delta_{(a-j)b} u_{x^jyx^c} (-u_{x^{a-j}yx^b})}\\
&-\sum_{j=0}^{c}{\delta_{b(c-j)} u_{x^ayx^j} u_{x^byx^{c-j}}}.
\end{align*}
We replace $c_w$ with the expression from equation \eqref{eq:cw} to find,
\begin{align*}
f_w=&\frac{1}{2}\Bigg(-u_{w}+(-1)^{\left|w\right|}u_{w}+2\sum_{j=0}^{b}{\delta_{aj} u_{x^ayx^j}u_{x^{b-j}yx^c}}\\
&+2\sum_{j=0}^{a}{\delta_{(a-j)b} u_{x^jyx^c} (-u_{x^{a-j}yx^b})}+2\sum_{j=0}^{c}{\delta_{b(c-j)}u_{x^ayx^j} u_{x^byx^{c-j}}}\Bigg)+ u_{w}\\
&-\sum_{j=0}^{b}{\delta_{aj} u_{x^ayx^j}u_{x^{b-j}yx^c}}-\sum_{j=0}^{a}{\delta_{(a-j)b} u_{x^jyx^c} (-u_{x^{a-j}yx^b})}-\sum_{j=0}^{c}{\delta_{b(c-j)} u_{x^ayx^j} u_{x^byx^{c-j}}}.
\end{align*}
Reordering the sums (and keeping in mind that $\left|w\right|$ is odd) gives,
\begin{align*}
f_w=&\frac{1}{2}\left(-u_{w}-u_{w}\right)+u_w+\frac{1}{2}2\underbrace{\sum_{j=0}^{b}{\delta_{aj} u_{x^ayx^j}u_{x^{b-j}yx^c}}}_{(\ast)}\\
&+\frac{1}{2}2\underbrace{\sum_{j=0}^{a}{\delta_{(a-j)b} u_{x^jyx^c} (-u_{x^{a-j}yx^b})}}_{(\ast\ast)}+\frac{1}{2}2\underbrace{\sum_{j=0}^{c}{\delta_{b(c-j)}u_{x^ayx^j} u_{x^byx^{c-j}}}}_{(\ast\ast\ast)}\\
&-\underbrace{\sum_{j=0}^{b}{\delta_{aj} u_{x^ayx^j}u_{x^{b-j}yx^c}}}_{(\ast)}-\underbrace{\sum_{j=0}^{a}{\delta_{(a-j)b} u_{x^jyx^c} (-u_{x^{a-j}yx^b})}}_{(\ast\ast)}-\underbrace{\sum_{j=0}^{c}{\delta_{b(c-j)}u_{x^ayx^j} u_{x^byx^{c-j}}}}_{(\ast\ast\ast)}\\
&=0,
\end{align*}
as the equally marked terms all cancel each other. We have found
\begin{equation*}
f_w=0.
\end{equation*}

\end{case1}

\begin{case2}
Consider now $w$ of even length. In this case,
\begin{align*}
f_{w}=&d_w + u_{w}-\sum_{j=0}^{b}{\delta_{aj} u_{x^ayx^j}u_{x^{b-j}yx^c}}\\
&-\sum_{j=0}^{a}{\delta_{(a-j)b} u_{x^jyx^c} (-u_{x^{a-j}yx^b})}-\sum_{j=0}^{c}{\delta_{b(c-j)} u_{x^ayx^j} u_{x^byx^{c-j}}}.
\end{align*}

Replacing $d_w$ with the expression from equation \eqref{eq:dw} yields,
\begin{align*}
f_{w}=&\frac{1}{2}\Bigg(-u_{w}+(-1)^{\left|w\right|}u_{w}+\underbrace{\sum_{j=0}^{b}{\delta_{aj} u_{x^ayx^j}u_{x^{b-j}yx^c}}}_{(\ast)}\\
&+\underbrace{\sum_{j=0}^{a}{\delta_{(a-j)b} u_{x^jyx^c} (-u_{x^{a-j}yx^b})}}_{(\ast\ast)}+\underbrace{\sum_{j=0}^{c}{\delta_{b(c-j)}u_{x^ayx^j} u_{x^byx^{c-j}}}}_{(\ast\ast\ast)}\Bigg)\\
&+ u_{w}-\underbrace{\sum_{j=0}^{b}{\delta_{aj} u_{x^ayx^j}u_{x^{b-j}yx^c}}}_{(\ast)}\\
&-\underbrace{\sum_{j=0}^{a}{\delta_{(a-j)b} u_{x^jyx^c} (-u_{x^{a-j}yx^b})}}_{(\ast\ast)}-\underbrace{\sum_{j=0}^{c}{\delta_{b(c-j)} u_{x^ayx^j} u_{x^byx^{c-j}}}}_{(\ast\ast\ast)}.
\end{align*}
We have marked equal terms in the equation above and find,
\begin{align*}
f_{w}=&\frac{1}{2}\left(-u_{w}+u_{w}\right)+ u_{w}\\
&-\frac{1}{2}\Bigg(\sum_{j=0}^{b}{\delta_{aj} u_{x^ayx^j}u_{x^{b-j}yx^c}}\\
&+\sum_{j=0}^{a}{\delta_{(a-j)b} u_{x^jyx^c} (-u_{x^{a-j}yx^b})}\\
&+\sum_{j=0}^{c}{\delta_{b(c-j)} u_{x^ayx^j} u_{x^byx^{c-j}}}\Bigg),
\end{align*}
which implies
\begin{equation*}
f_w=u_w-\frac{1}{2}\Bigg(\sum_{j=0}^{b}{\delta_{aj} u_{x^ayx^j}u_{x^{b-j}yx^c}}+\sum_{j=0}^{a}{\delta_{(a-j)b} u_{x^jyx^c} (-u_{x^{a-j}yx^b})}+\sum_{j=0}^{c}{\delta_{b(c-j)} u_{x^ayx^j} u_{x^byx^{c-j}}}\Bigg).
\end{equation*}
\end{case2}
\end{degree2}

\subsection{Proof of Theorem \ref{thm:OH}}
Finally, we are in a position to prove our main theorem.
\begin{proof}[Proof of Theorem \ref{thm:OH}]
The word $w=x^2yx^4y$ is of even length. Thus the coefficient $f_w$ is given by the following formula.

\begin{align*}
f_w=&u_{w}-\frac{1}{2}\Bigg(\sum_{j=0}^{4}{\delta_{2j} u_{x^2yx^j}u_{x^{4-j}y}}+\sum_{j=0}^{2}{\delta_{(2-j)4} u_{x^jy} (-u_{x^{2-j}yx^4})}\\
&+\sum_{j=0}^{0}{\delta_{4(0-j)} u_{x^2yx^j} u_{x^4yx^{0-j}}}\Bigg)\\
=&u_{w}-\frac{1}{2}\big(\delta_{20} u_{x^2y}u_{x^{4}y}+\underbrace{\delta_{21}}_{=0} u_{x^2yx}u_{x^{3}y}+\delta_{22} u_{x^2yx^2}u_{x^{2}y}+\underbrace{\delta_{23}}_{=0} u_{x^2yx^3}u_{xy}+\delta_{24} u_{x^2yx^4}\underbrace{u_{y}}_{=0}\\
&+\delta_{24} \underbrace{u_{y}}_{=0} (-u_{x^{2}yx^4})+\underbrace{\delta_{14}}_{=0} u_{xy} (-u_{xyx^4})+\delta_{04} u_{x^2y} (-u_{yx^4})+\delta_{40} u_{x^2y} u_{x^4y}\big).
\end{align*}
The marked terms equal 0. The expression simplifies to,
\begin{equation*}
f_w=u_{w}-\frac{1}{2}u_{x^2y}\big(2u_{x^{4}y}+u_{x^2yx^2}-u_{yx^4}\big).
\end{equation*}
Now,
\begin{align*}
u_{x^2yx^4y}&=\frac{1}{(2\pi i)^8}\zeta(3,5)\\
u_{x^2y}&=\frac{-1}{(2\pi i)^3}\zeta(3)\\
u_{x^{4}y}&=\frac{-1}{(2\pi i)^5}\zeta(5)\\
u_{x^2yx^2}&=\frac{-1}{(2\pi i)^3}\zeta(x^2yx^2)=\frac{-1}{(2\pi i)^5}(-1)^2{4\choose 2}\zeta(5)=\frac{-6}{(2\pi i)^5}\zeta(5)\\
u_{yx^4}&=\frac{-1}{(2\pi i)^5}\zeta(yx^4)=\frac{-1}{(2\pi i)^5}(-1)^0{4\choose 4}\zeta(5)=\frac{-1}{(2\pi i)^5}\zeta(5).
\end{align*}
The result follows via,
\begin{align*}
f_{w}&=u_{w}-\frac{1}{2}u_{x^2y}\big(2u_{x^{4}y}+u_{x^2yx^2}-u_{yx^4}\big)\\
&=\frac{\zeta(3,5)}{(2\pi i)^8}-\frac{1}{2}\frac{(-\zeta(3))}{(2\pi i)^3}\bigg(-2\frac{\zeta(5)}{(2\pi i)^5}-6\frac{\zeta(5)}{(2\pi i)^5}+\frac{\zeta(5)}{(2\pi i)^5}\bigg)\\
&=\frac{1}{(2\pi i)^8}\big(\zeta(3,5)-\frac{7}{2}\zeta(3)\zeta(5)\big)\\
&=\frac{1}{2(2\pi i)^8}\big(2\zeta(3,5)-7\zeta(3)\zeta(5)\big)\\
&=\cfrac{2\zeta(3,5)-7\zeta(3)\zeta(5)}{512\pi^8}.
\end{align*}

As for the irrationality of $f_w$, first note that assuming the conjectured basis to be in fact a basis over $\mathbb{Q}$ implies: If 
\begin{equation*}
q_1\zeta(3,5)+q_2\zeta(3)\zeta(5)+q_3\zeta(3)^2\zeta(2)+q_4\zeta(2)^4=0,
\end{equation*}
for $q_1,q_2,q_3,q_4\in \mathbb{Q}$, then $q_1=q_2=q_3=q_4=0$. Now, assume on the contrary that $f_w=q\in \mathbb{Q}$. We find,
\begin{align*}
&\cfrac{2\zeta(3,5)-7\zeta(3)\zeta(5)}{512\pi^8}=q\\
\Leftrightarrow& 2\zeta(3,5)-7\zeta(3)\zeta(5)-512\pi^8 q=0.
\end{align*}
Since $\zeta(2)=\frac{\pi^2}{6}$, we have $\pi^8=6^4\zeta(2)^4=1296\zeta(2)^4$. Hence,
\begin{equation*}
2\zeta(3,5)-7\zeta(3)\zeta(5)-512\cdot q\cdot 1296\zeta(2)^4=0.
\end{equation*}
But this contradicts the statement above . Therefore, if the conjecture is true, the coefficient has to be irrational.
\end{proof}

\section{The Alekseev-Torossian associator}
There is an alternative and very useful description of  $\psi\in \GRT_1$  which can be found in \cite{Rossi2014}, namely,
\begin{equation}
\psi(x,y)=\mathcal{T}\exp\left(\int\limits^{1}_{0}{x(s)ds}\right).
\end{equation}
Here
\begin{equation}
x(s):=\sum\limits_{j=1}^{\infty}x_{2j+1}(s(s-1))^{2j}\in \grt_1
\end{equation}
for elements $x_{2j+1}\in \grt_1$ indexed by the corresponding degree in the grading of $\hat{\mathbb{F}}_{\Lie}(x,y)\supset\grt_1$. This means that $x_{2j+1}$ will be a linear combination of Lie words in $x$ and $y$ having $2j$ brackets. They are determined by the equation $\psi\cdot\PhiKZ=\PhiAKZ$.
The path-ordered exponential 
\begin{equation*}
\mathcal{T}\exp\left(\int\limits^{1}_{0}{x(s)ds}\right)
\end{equation*}
is defined as
\begin{equation}
\mathcal{T}\exp\left(\int\limits^{1}_{0}{x(s)ds}\right)=1+\int\limits^{1}_{0}{x(s)ds}+\int\limits^{1}_{0}{x(s_1)\int\limits^{s_1}_{0}{x(s_2)ds_2}ds_1}+\dots
\end{equation}

\begin{lemma}\cite{Rossi2014}
Let $\tpsi\in \grt_1$ be given by,
\begin{equation}
\tpsi(x,y):=\mathcal{T}exp\left(\int\limits_{0}^{\frac{1}{2}}{x(s)ds}\right).
\end{equation}
By acting with $\tpsi$ on $\PhiKZ$, we obtain a Drinfeld associator called the \textit{Alekseev-Torossian associator} $\PhiAT$, i.e.
\begin{equation}
\PhiAT(x,y):=(\tpsi\cdot \PhiKZ)(x,y).
\end{equation}
\end{lemma}
This associator was first defined in the work of A. Alekseev and C. Torossian \cite{Alekseev2010}. As for $\PhiOH$, we ask whether $\PhiAT$ will be an associator with only rational coefficients. Unfortunately, again, this does not seem to be the case. Our second major result is the following theorem.
\begin{theorem}\label{thm:AT}
The coefficient in $\PhiAT$ of the word $w=x^2yx^4y$ is
\begin{equation*}
\cfrac{2048\zeta(3,5)-6293\zeta(3)\zeta(5)}{524288\pi^8}.
\end{equation*}
If the conjectured basis $\left\{\zeta(3,5),\zeta(3)\zeta(5),\zeta(3)^2\zeta(2),\zeta(2)^4\right\}$ from conjecture \ref{conjecture} is in fact a $\mathbb{Q}$-basis for multiple zeta values of weight 8, then this coefficient is irrational.
\end{theorem}

\subsection{Towards a proof of Theorem \ref{thm:AT}}
Our first aim is to describe the elements $x_{2j+1} \in \grt_1\subset \hat{\mathbb{F}}_{\Lie}(x,y)$. For this we have to understand what the space of linear combinations of Lie words having $2j$ brackets looks like. This is where the notion of so-called Lyndon bases becomes useful. The main reference here is a book on free Lie algebras by C. Reutenauer \cite{Reutenauer}.
\begin{definition}
A word in $x,y$ which is the unique minimal element with respect to the lexicographical ordering within the set of its rotations is called \textit{Lyndon word}.
\end{definition}

\begin{remark}\cite{Reutenauer}
There is a bijection $\gamma$ between Lyndon words and a basis of the free Lie algebra. Let $w$ be a Lyndon word. If $w$ has length 1, set $\gamma(w)=w$. If $\left|w\right|\geq 2$, we may decompose $w$ as $w=uv$, where $u,v$ are again Lyndon words and $v$ is of maximal length. Thus define $\gamma(w)=\left[\gamma(u),\gamma(v)\right]$, recursively. 
\end{remark}

We will use this to find linear generators for the spaces of Lie words of degree 1 and 2 (in $y$). In degree 1, the Lyndon words are of the form $yx^a$ for $a\in \mathbb{N}_0$. The map $\gamma$ sends such a word to $\pm\ad_x^{a}(y)$, depending on the parity of $a$. Here $\ad_x(y)=\left[x,y\right]$. Therefore, the set $\left\{\ad_x^a(y)\right\}_{a\in\mathbb{N}_0}$ describes a basis for the vector space of Lie words of degree 1. 

In degree 2, the Lyndon words are given by $yx^{\alpha}yx^{\beta}$ with $\alpha<\beta, \alpha,\beta\in \mathbb{N}_0$. Decomposing $w=yx^{\alpha}yx^{\beta}$ as $w=uv$, $u=yx^{\alpha},v=yx^{\beta}$, we have that 
\begin{equation}
\gamma(w)=\left[\gamma(yx^{\alpha}),\gamma(yx^{\beta})\right].
\end{equation}
Using the definition of $\gamma$ on degree 1 words, we find
\begin{equation}
\gamma(w)=\pm\left[\ad_x^{\alpha}(y),\ad_x^{\beta}(y)\right].
\end{equation}
This way, $\left\{\left[\ad_x^{\alpha}(y),\ad_x^{\beta}(y)\right]\right\}_{\alpha<\beta}$ generates the space of Lie words of degree 2 linearly.

Let us fix some $n\geq 1$ and consider $x_{2n+1}\in\hat{\mathbb{F}}_{\Lie}(x,y)_{2n+1}$ (the linear span of Lie words in $2n$ brackets). Using the above bases, we may describe $x_{2n+1}$ up to degree 2. It will be of the form,
\begin{equation}
x_{2n+1}=c_{2n}\ad_{x}^{2n}(y)+\sum_{\alpha,\beta}{c_{\alpha,\beta}\left[\ad_x^{\alpha}(y),\ad_x^{\beta}(y)\right]}+\dots
\end{equation}
Here, $c_{2n}$, $c_{\alpha,\beta}\in \mathbb{K}$, $\alpha \in \mathbb{N}_{0}$, $\beta\in \mathbb{N}$, $\alpha<\beta$ such that $\alpha+\beta+2=2n+1$. Using this notation the first few terms of $\psi$ are thus given by
\begin{align*}
&\psi(x,y)=\mathcal{T}exp\left(\int\limits_{0}^{1}{x(s)ds}\right)=1+\int\limits_{0}^{1}{\sum\limits_{j=1}^{\infty}{x_{2j+1}}(s(s-1))^{2j}ds}+\dots\\
=&1+\sum\limits_{j=1}^{\infty}{\Big(c_{2j}\ad_{x}^{2j}(y)+\sum\limits_{\substack{0\leq\alpha<\beta\\ \alpha+\beta+2=2j+1}}{c_{\alpha,\beta}\left[\ad_x^{\alpha}(y),\ad_x^{\beta}(y)\right]}+\dots\Big)\int\limits_{0}^{1}{(s(s-1))^{2j}ds}}+\dots
\end{align*}

\begin{remark}
To pass from $\hat{\mathbb{F}}_{\Lie}(x,y)$ to $\mathbb{K}\left\langle \left\langle x,y \right\rangle\right\rangle$, one sets $\left[x,y\right]=xy-yx$. Below, we give some useful identities.
\begin{align*}
\ad_{x}^{2n}(y)&=\sum\limits_{i=0}^{2n}{{2n\choose i}(-1)^{i}x^iyx^{2n-i}}\\
\ad_{y}\ad_{x}^{2n}(y)&=\sum\limits_{i=0}^{2n}{{2n\choose i}(-1)^{i}(yx^iyx^{2n-i}-x^iyx^{2n-i}y)}\\
\left[\ad_x^{\alpha}(y),\ad_x^{\beta}(y)\right]&=\sum\limits_{j=0}^{\alpha}{\sum\limits_{i=0}^{\beta}{{\alpha \choose j}{\beta\choose i}(-1)^{j+i}\left(x^jyx^{\alpha-j+i}yx^{\beta-i}-x^iyx^{\beta-i+j}yx^{\alpha-j}\right)}}
\end{align*}
Moreover, we shall abbreviate the following integral by $I_{1}^{2n}$.
\begin{equation}\label{eq:integral}
I_{1}^{2n}:=\int\limits_{0}^{1}{(s(s-1))^{2n}ds}=\frac{(\Gamma(2n+1))^2}{\Gamma(4n+2)}=\frac{\left((2n)!\right)^2}{(4n+1)!}.
\end{equation}
Here $\Gamma(x)=\int\limits_{0}^{\infty}{e^{-u}u^{x-1}du}$ is the gamma function (see for instance \cite{Artin64}). 
\end{remark}

\begin{lemma}
The coefficient $c_{2n}$ is
\begin{equation}
c_{2n}=\frac{2(4n+1)!\zeta(2n+1)}{(2\pi i)^{2n+1}\left((2n)!\right)^2}.
\end{equation}
\end{lemma}

\begin{proof}
To find $c_{2n}$, we have to consider terms of degree 1 (in $y$). These are all of the form $x^iyx^{2n-i}$ for some $i\in\left\{0,\dots,2n\right\}$ (as the total number of $x$'s and $y$'s has to be $2n+1$). In the defining equation $\psi\cdot\PhiKZ=\PhiAKZ$, these terms appear as follows,
\begin{align*}
&(1+\frac{\left((2n)!\right)^2}{(4n+1)!} c_{2n}\ad_{x}^{2n}(y)+\dots)(1+\sum\limits_{i=0}^{2n}{u_{x^iyx^{2n-i}}x^iyx^{2n-i}}+\dots)\\
=&(1+\frac{\left((2n)!\right)^2}{(4n+1)!} c_{2n}\sum\limits_{i=0}^{2n}{{2n\choose i}(-1)^{i}x^iyx^{2n-i}}+\dots)(1+\sum\limits_{i=0}^{2n}{u_{x^iyx^{2n-i}}x^iyx^{2n-i}}+\dots)\\
=&-\sum\limits_{i=0}^{2n}{u_{x^iyx^{2n-i}}x^iyx^{2n-i}}+\dots
\end{align*}
This is equivalent to saying that for all $j\in\left\{0,\dots,2n\right\}$,
\begin{equation}\label{eq:c2n}
\frac{\left((2n)!\right)^2}{(4n+1)!} c_{2n}{2n\choose j}(-1)^{j}=-2u_{x^jyx^{2n-j}}.
\end{equation}
Recall that $u_{x^jyx^{2n-j}}=\frac{-1}{(2\pi i)^{2n+1}}\zeta(x^jyx^{2n-j})$, where $\zeta(w)$ is the regularized multiple zeta value of the word $w$. By regularizing, we find (see Lemma \ref{lemma:zeta}),
\begin{equation*}
\zeta(x^jyx^{2n-j})=(-1)^{2n-j}{2n\choose j}\zeta(2n+1).
\end{equation*}
The above equation \eqref{eq:c2n} becomes,
\begin{align*}
\frac{\left((2n)!\right)^2}{(4n+1)!} c_{2n}{2n\choose j}(-1)^{j}&=-2\frac{-1}{(2\pi i)^{2n+1}}(-1)^{2n-j}{2n\choose j}\zeta(2n+1)\\
\Leftrightarrow\frac{\left((2n)!\right)^2}{(4n+1)!} c_{2n}&=2\frac{1}{(2\pi i)^{2n+1}}\zeta(2n+1)
\end{align*}
from which the statement follows.
\end{proof}

\begin{lemma}
Fix $n\geq 1$ and let $a,b,c\in\mathbb{N}_{0}$ such that $a+b+c+2=2n+1$. The coefficients $c_{\alpha,\beta}$ satisfy the following set of equations.
\begin{align}
\nonumber -2u_{x^ayx^byx^c}=&I_{1}^{2n}\sum\limits_{\substack{0\leq\alpha<\beta\\  \alpha+\beta+2=2n+1}}{c_{\alpha,\beta}\left({\alpha\choose a}{\beta\choose c}(-1)^{a+\beta-c}-{\alpha\choose c}{\beta\choose a}(-1)^{a+\alpha-c}\right)}\\ 
\label{eq:cab} +&\sum\limits_{\substack{p\in \mathbb{N}_0,s\in \mathbb{N}\\2s=b+a-p}}{I_1^{2s} c_{2s}\left({2s\choose a}(-1)^a-{2s\choose b}(-1)^{b}\right) u_{x^pyx^c}}\\
\nonumber +&\sum\limits_{\substack{q \in \mathbb{N}_0, s\in \mathbb{N}\\ 2s=b+c-q}}{I_{1}^{2s}c_{2s}{2s\choose b}(-1)^b u_{x^ayx^q}}.
\end{align}

\end{lemma}

\begin{proof}
Let $n,a,b,c$ be as in the lemma and set $w=x^ayx^byx^c$. This word of degree 2 (in $y$) will appear within the following terms of the defining equation $\psi\cdot \PhiKZ=\PhiAKZ$,
\begin{align*}
&\Bigl(\Bigr. 1+I_{1}^{2n}\sum_{\alpha,\beta}{c_{\alpha,\beta}\left[\ad_x^{\alpha}(y),\ad_x^{\beta}(y)\right]}
+\sum\limits_{\substack{s \\ 2s+p+q+2=2n+1}}{I_{1}^{2s}c_{2s}\ad_{x}^{2s}(y)}+\dots\Bigl.\Bigr)\\
&\cdot\Bigl(\Bigr. 1+u_{x^ayx^byx^c}x^ayx^byx^c+\sum\limits_{\substack{p,q \\ 2s+p+q+2=2n+1}}{u_{x^pyx^q}x^pyx^q}+\dots\Bigl.\Bigr)\\
=&-u_{x^ayx^byx^c}x^ayx^byx^c+\dots
\end{align*}

Let us try to pick out the terms giving $x^ayx^byx^c$ in the product on the left.
\begin{enumerate}[label=(\roman*)]
\item{The simplest one is $1\cdot u_{x^ayx^byx^c}x^ayx^byx^c$.}
\item{Consider
\begin{align*}
&I_{1}^{2n}\sum\limits_{\alpha,\beta}{c_{\alpha,\beta}\left[\ad_x^{\alpha}(y),\ad_x^{\beta}(y)\right]}\cdot 1\\
=&I_{1}^{2n}\sum\limits_{\alpha,\beta}{c_{\alpha,\beta}\sum\limits_{j=0}^{\alpha}{\sum\limits_{i=0}^{\beta}{{\alpha \choose j}{\beta\choose i}(-1)^{j+i}\left(x^jyx^{\alpha-j+i}yx^{\beta-i}-x^iyx^{\beta-i+j}yx^{\alpha-j}\right)}}}\cdot 1
\end{align*}
For fixed $\alpha$ and $\beta$, this sum contributes  with the first term $x^jyx^{\alpha-j+i}yx^{\beta-i}$ when $j=a$ and $i=\beta-c$, and with the second term $-x^iyx^{\beta-i+j}yx^{\alpha-j}$ whenever $j=\alpha-c$ and $i=a$. Thus the coefficients we have to consider are given by
\begin{equation*}
I_{1}^{2n}\sum\limits_{\substack{0\leq\alpha<\beta\\ \alpha+\beta+2=2n+1}}{c_{\alpha,\beta}\left({\alpha\choose a}{\beta\choose c}(-1)^{a+\beta-c}-{\alpha\choose c}{\beta\choose a}(-1)^{a+\alpha-c}\right)}.
\end{equation*}
}
\item{The last terms we need to deal with are given by the action of $I_{1}^{2s}c_{2s}\ad_{x}^{2s}(y)$ on elements of degree 1 (in $y$) via,
\begin{equation*}
I_{1}^{2s}c_{2s}\ad_{x}^{2s}(y)\cdot (x^pyx^q)+\left[y,I_{1}^{2s}c_{2s}\ad_{x}^{2s}(y)\right]\partial_{y}(x^pyx^q)
\end{equation*}
for $s\in \mathbb{N}, p,q \in \mathbb{N}_{0}$ such that $2s+p+q+2=2n+1$. Expanding the left product in this expression gives,
\begin{align*}
&I_{1}^{2s}c_{2s}\ad_{x}^{2s}(y)\cdot (x^pyx^q)\\
=&I_{1}^{2s}c_{2s}\sum\limits_{i=0}^{2s}{{2s\choose i}(-1)^{i}x^iyx^{2s-i}}\cdot (x^pyx^q)\\
=&I_{1}^{2s}c_{2s}\sum\limits_{i=0}^{2s}{{2s\choose i}(-1)^{i}x^iyx^{2s-i+p}yx^q}
\end{align*}
This contributes when $i=a$ and $q=c$. Therefore, the coefficients we are interested in are of the form,
\begin{equation*}
\sum\limits_{\substack{p\in \mathbb{N}_0,s\in \mathbb{N}\\2s=b+a-p}}{I_1^{2s} c_{2s}{2s\choose a}(-1)^a u_{x^pyx^c}}.
\end{equation*}
The right part of the action of $I_{1}^{2s}c_{2s}\ad_{x}^{2s}(y)$ on $x^pyx^q$ expands to,
\begin{align*}
&\left[y,I_{1}^{2s}c_{2s}\ad_{x}^{2s}(y)\right]\partial_{y}(x^pyx^q)\\
=&I_{1}^{2s}c_{2s}x^p\left[y,\ad_x^{2s}(y)\right]x^q\\
=&I_{1}^{2s}c_{2s}x^p\ad_y\ad_x^{2s}(y)x^q\\
=&I_{1}^{2s}c_{2s}x^p\sum\limits_{i=0}^{2s}{{2s\choose i}(-1)^{i}(yx^iyx^{2s-i}-x^iyx^{2s-i}y)}x^q\\
=&I_{1}^{2s}c_{2s}\sum\limits_{i=0}^{2s}{{2s\choose i}(-1)^{i}(x^pyx^iyx^{2s-i+q}-x^{p+i}yx^{2s-i}yx^q)}.
\end{align*}
This contributes with the first term $x^pyx^iyx^{2s-i+q}$ when $p=a$ and $i=b$, and with the second term $-x^{p+i}yx^{2s-i}yx^q$ when $q=c$ and $i=a-p$. The corresponding coefficients are thus,
\begin{equation*}
\sum\limits_{\substack{q \in \mathbb{N}_0, s\in \mathbb{N}\\ 2s=b+c-q}}{I_{1}^{2s}c_{2s}{2s\choose b}(-1)^b u_{x^ayx^{q}}}-\sum\limits_{\substack{p\in \mathbb{N}_0,s\in \mathbb{N}\\2s=b+a-p}}{I_{1}^{2s}c_{2s}{2s\choose a-p}(-1)^{a-p} u_{x^{p}yx^c}}.
\end{equation*}
Note that in the last sum, $a-p=2s-b$ for all $s$ and thus we may replace ${2s\choose a-p}$ by ${2s\choose b}$ and $(-1)^{a-p}$ by $(-1)^b$.
}
\end{enumerate}

Hence, in summary, comparing the coefficients of $w=x^ayx^byx^c$ with $a+b+c+2=2n+1$ in the equation $\psi\cdot\PhiKZ=\PhiAKZ$ yields the following relation for the $c_{\alpha,\beta}$.
\begin{equation}
\begin{aligned}
-2u_{x^ayx^byx^c}\stackrel{!}{=}&I_{1}^{2n}\sum\limits_{\substack{0\leq\alpha<\beta\\ \alpha+\beta+2=2n+1}}{c_{\alpha,\beta}\left({\alpha\choose a}{\beta\choose c}(-1)^{a+\beta-c}-{\alpha\choose c}{\beta\choose a}(-1)^{a+\alpha-c}\right)}\\
&+\sum\limits_{\substack{p\in \mathbb{N}_0,s\in \mathbb{N}\\2s=b+a-p}}{I_1^{2s} c_{2s}\left({2s\choose a}(-1)^a-{2s\choose b}(-1)^{b}\right) u_{x^pyx^c}}\\
&+\sum\limits_{\substack{q \in \mathbb{N}_0, s\in \mathbb{N}\\ 2s=b+c-q}}{I_{1}^{2s}c_{2s}{2s\choose b}(-1)^b u_{x^ayx^q}}.
\end{aligned}
\end{equation}
\end{proof}

The condition $a+b+c+2=2n+1$ gives us $\frac{(2n+1)2n}{2}$ linear equations in the $n$ unknowns $c_{\alpha,\beta}$. By implementing the above equations for $n\leq 5$, and thereafter solving them numerically, we obtain numerical approximations of the $c_{\alpha,\beta}$. More important than the actual value of these numbers is the fact, that a computer is able to solve the overdetermined system of equations above. This tells us that the cumbersome algebraic manipulations performed so far have a good chance of being correct. 

In order to give a formula for the coefficients of $\PhiAT$, we need the following notion.
\begin{definition}
The \textit{incomplete beta function} $B_{x}(a,b)$ is defined as
\begin{equation*}
B_{x}(a,b)=\int\limits_{0}^{x}{t^{a-1}(1-t)^{b-1} dt}.
\end{equation*}
\end{definition}
When $x=1$, $B_1(a,b)=:B(a,b)$ describes the usual beta function \cite{Artin64}.
\begin{lemma}
The quotient $I_{x}(a,b):=\frac{B_x(a,b)}{B(a,b)}$ satisfies the relation
\begin{equation*}
I_{x}(a,b)=1-I_{1-x}(b,a)
\end{equation*}
\end{lemma}
\begin{proof}
The relation is obtained by substituting the variable $t$ with $1-u$ in the definition.
\begin{align*}
I_{x}(a,b)&=\frac{\int\limits_{0}^{x}{t^{a-1}(1-t)^{b-1} dt}}{\int\limits_{0}^{1}{t^{a-1}(1-t)^{b-1} dt}}=\frac{\int\limits_{1-x}^{1}{(1-u)^{a-1} u^{b-1} du}}{\int\limits_{0}^{1}{(1-u)^{a-1} u^{b-1} du}}\\
&=\frac{\int\limits_{1-x}^{1}{(1-u)^{a-1} u^{b-1} du}+\int\limits_{0}^{1-x}{(1-u)^{a-1} u^{b-1} du}-\int\limits_{0}^{1-x}{(1-u)^{a-1} u^{b-1} du}}{\int\limits_{0}^{1}{(1-u)^{a-1} u^{b-1} du}}\\
&=1-\frac{\int\limits_{0}^{1-x}{(1-u)^{a-1} u^{b-1} du}}{\int\limits_{0}^{1}{(1-u)^{a-1} u^{b-1} du}}=1-I_{1-x}(b,a).
\end{align*}
\end{proof}
Using this fact, we can easily prove the following
\begin{lemma}\label{lemma:integral}
The integral $\int\limits_{0}^{\frac{1}{2}}{(s(s-1))^{2n}ds}$ which we shall abbreviate by $J_1^{2n}$ has the value
\begin{equation}
J_{1}^{2n}=\frac{\left((2n)!\right)^2}{2(4n+1)!}
\end{equation}
\end{lemma}
\begin{proof}
We use the above identity to get,
\begin{align*}
\frac{J_{1}^{2n}}{I_{1}^{2n}}&=\frac{B_{\frac{1}{2}}(2n+1,2n+1)}{B(2n+1,2n+1)}=1-\frac{B_{1-\frac{1}{2}}(2n+1,2n+1)}{B(2n+1,2n+1)}=1-\frac{J_{1}^{2n}}{I_{1}^{2n}}\\
\Leftrightarrow J_{1}^{2n}&=\frac{1}{2}I_{1}^{2n}=\frac{\left((2n)!\right)^2}{2(4n+1)!}.
\end{align*}
\end{proof}


We now have all the tools to describe the coefficients of the Alekseev-Torossian associator $\PhiAT$. For this, write 
\begin{equation*}
\PhiAT(x,y):=1+\sum\limits_{w; \left|w\right|\geq 2}{\tilde{f}_w w}
\end{equation*}
and set 
\begin{equation}\label{abbr}
J_{2}^{l,m}:=\int\limits_{0}^{\frac{1}{2}}{(s_1(s_1-1))^{2l}\int\limits^{s_1}_{0}{(s_2(s_2-1))^{2m}ds_2}ds_1}.
\end{equation}
\begin{lemma}
The coefficients of words in degree $1$ and $2$ of $\PhiAT$ are given by
\begin{align}
\label{eq:atdeg1}\tilde{f}_{x^ayx^b}=&
\begin{cases} 
      0 & a+b+1=2n+1\\
      u_{x^ayx^b} & a+b+1=2n
\end{cases}\\
\label{eq:atdeg2}\tilde{f}_{x^ayx^byx^c}=&
\begin{cases}
0  & a+b+c+2=2n+1\\
u_{x^ayx^byx^c}+K(a,b,c) & a+b+c+2=2n
\end{cases}
\end{align}
where 
\begin{align*}
K(a,b,c)=&\sum\limits_{\substack{l,m\\2l+2m+2=2n}}J_2^{l,m}c_{2l}c_{2m}\biggl({2l\choose a}{2m\choose c}(-1)^{a-c}\\
+&{2m\choose a}{2l\choose b}(-1)^{a+b}-{2m\choose c}{2l\choose b}(-1)^{b+c}\biggr)\\
+&\sum\limits_{\substack{p\in \mathbb{N}_0,s\in \mathbb{N}\\2s=b+a-p}}{\frac{(-1)^{c+1}(p+c)!}{(2\pi i)^{2n}p!c!}\zeta(2s+1)\zeta(p+c+1)\left({2s\choose a}(-1)^a-{2s\choose b}(-1)^{b}\right)}\\
+&\sum\limits_{\substack{q \in \mathbb{N}_0, s\in \mathbb{N}\\ 2s=b+c-q}}{\frac{(-1)^{q+1}(a+q)!}{(2\pi i)^{2n}a!q!}\zeta(2s+1)\zeta(a+q+1){2s\choose b}(-1)^b }.
\end{align*}
\end{lemma}

\begin{proof}
Recall that $\tpsi$ is given by
\begin{equation*}
\tpsi(x,y):=\mathcal{T}exp\left(\int\limits_{0}^{\frac{1}{2}}{x(s)ds}\right).
\end{equation*}
Extending this element yields,
\begin{align*}
\tpsi(x,y)=&\mathcal{T}\exp\left(\int\limits^{\frac{1}{2}}_{0}{x(s)ds}\right)=1+\int\limits^{\frac{1}{2}}_{0}{x(s)ds}+\int\limits^{\frac{1}{2}}_{0}{x(s_1)\int\limits^{s_1}_{0}{x(s_2)ds_2}ds_1}+\dots\\
=&1+\int\limits_{0}^{\frac{1}{2}}{\sum\limits_{j=1}^{\infty}{x_{2j+1}(s(s-1))^{2j}ds}}\\
&+\int\limits^{\frac{1}{2}}_{0}{\sum\limits_{l=1}^{\infty}{x_{2l+1}(s_1(s_1-1))^{2l}}\int\limits^{s_1}_{0}{\sum\limits_{m=1}^{\infty}{x_{2m+1}(s_2(s_2-1))^{2m}}ds_2}ds_1}+\dots\\
=&1+\sum\limits_{j=1}^{\infty}{x_{2j+1}\int\limits_{0}^{\frac{1}{2}}{(s(s-1))^{2j}ds}}\\
&+\sum\limits_{l=1}^{\infty}{\sum\limits_{m=1}^{\infty}{x_{2l+1}x_{2m+1}\int\limits^{\frac{1}{2}}_{0}{(s_1(s_1-1))^{2l}}\int\limits^{s_1}_{0}{(s_2(s_2-1))^{2m}}ds_2}ds_1}+\dots\\
=&1+\sum\limits_{j=1}^{\infty}{x_{2j+1}J_1^{2j}}+\sum\limits_{l=1}^{\infty}{\sum\limits_{m=1}^{\infty}{x_{2l+1}x_{2m+1}J_2^{l,m}}}+\dots
\end{align*}
where $J_{2}^{l,m}$ is as in equation \eqref{abbr}.
Replacing $x_{2j+1}$ gives,
\begin{align*}
\tpsi(x,y)=&1+\sum\limits_{j=1}^{\infty}{J_1^{2j}\Big(c_{2j}\ad_{x}^{2j}(y)+\sum\limits_{\substack{0\leq\alpha<\beta\\ \alpha+\beta+2=2j+1}}{c_{\alpha,\beta}\left[\ad_x^{\alpha}(y),\ad_x^{\beta}(y)\right]}+\dots\Big)}\\
&+\sum\limits_{l=1}^{\infty}{\sum\limits_{m=1}^{\infty}{x_{2l+1}x_{2m+1}J_2^{l,m}}}+\dots.
\end{align*}

We prove formula \eqref{eq:atdeg1}. Let $w=x^ayx^b$ with $a,b\in \mathbb{N}_{0}$ such that $\left|w\right|=a+b+1=2n+1$. To find its coefficient, we have to consider the product,
\begin{align*}
(\tpsi\cdot\PhiKZ)(x,y)&=(1+J_1^{2n}x_{2n+1}+\dots)(1+u_{x^ayx^b}x^ayx^b+\dots)\\
&=(1+J_1^{2n}(c_{2n}\ad_x^{2n}(y)+\dots)+\dots)(1+u_{x^ayx^b}x^ayx^b+\dots)\\
&=1+J_1^{2n}c_{2n}\sum\limits_{i=0}^{2n}{{2n\choose i}(-1)^{i}x^iyx^{2n-i}}\cdot 1+1\cdot u_{x^ayx^b}x^ayx^b+\dots\\
&=1+J_{1}^{2n}c_{2n}(-1)^a{2n \choose a}x^ayx^b+u_{x^ayx^b}x^ayx^b+\dots\\
&\stackrel{!}{=}1+\tilde{f}_{x^ayx^b}x^ayx^b+\dots=\PhiAT(x,y)
\end{align*}
The coefficient $\tilde{f}_w$ is therefore given by,
\begin{equation*}
\tilde{f}_w=u_w+J_{1}^{2n}c_{2n}(-1)^a{2n \choose a}=u_w-2\frac{1}{2}u_{w}=0,
\end{equation*}
where we used,
\begin{equation*}
\frac{\left((2n)!\right)^2}{(4n+1)!} c_{2n}(-1)^{a}{2n\choose a}=-2u_{x^ayx^{b}}
\end{equation*}
which is exactly equation \eqref{eq:c2n}.

Let us consider now a word of even length, i.e. $w=x^ayx^b$ with $a,b\in \mathbb{N}_{0}$ such that $a+b+1=2n$. Then $\tpsi$ does not contribute anything to the coefficients and we have,
\begin{equation*}
\tilde{f}_w=u_{w}.
\end{equation*}

To show \eqref{eq:atdeg2}, note that a word $w=x^ayx^byx^c$ of odd length, i.e. $a,b,c \in \mathbb{N}_{0}$ with $a+b+c+2=2n+1$, appears in the following sum. We have labeled the terms with their origin within the product $\tpsi\cdot\PhiKZ$.
\begin{align*}
&\underbrace{1}_{\tpsi}\cdot \underbrace{\sum\limits_{\substack{\tilde{w};deg(\tilde{w})=2\\\left|\tilde{w}\right|=2n+1}}{u_{\tilde{w}} \tilde{w}}}_{\PhiKZ}+\underbrace{\sum\limits_{\substack{0\leq \alpha<\beta\\ \alpha+\beta+2=2n+1}}{J_1^{2n}c_{\alpha,\beta}\left[\ad_x^{\alpha}(y),\ad_x^{\beta}(y)\right]}}_{\tpsi}\cdot \underbrace{1}_{\PhiKZ}\\
&+\sum\limits_{\substack{s\\2s+p+q+2=2n+1}}\sum\limits_{\substack{p,q\\v=x^pyx^q}}{\left(\underbrace{J_1^{2s}c_{2s}\ad_x^{2s}(y)}_{\tpsi} \underbrace{u_v v}_{\PhiKZ}+[y,\underbrace{J_1^{2s}c_{2s}\ad_x^{2s}}_{\tpsi}]\partial_y(\underbrace{u_v v}_{\PhiKZ})\right)}.
\end{align*}
With a similar reasoning as in the case where we determined the $x_{2j+1}$'s, we obtain,
\begin{align*}
\tilde{f}_w=&u_w+J_{1}^{2n}\sum\limits_{\substack{0\leq\alpha<\beta \\ \alpha+\beta+2=2n+1}}{c_{\alpha,\beta}\left({\alpha\choose a}{\beta\choose c}(-1)^{a+\beta-c}-{\alpha\choose c}{\beta\choose a}(-1)^{a+\alpha-c}\right)}\\
&+\sum\limits_{\substack{p\in \mathbb{N}_0,s\in \mathbb{N}\\2s=b+a-p}}{J_1^{2s} c_{2s}u_{x^pyx^c}\left({2s\choose a}(-1)^a-{2s\choose b}(-1)^{b}\right)}\\
&+\sum\limits_{\substack{q \in \mathbb{N}_0, s\in \mathbb{N}\\ 2s=b+c-q}}{J_{1}^{2s}c_{2s}u_{x^ayx^q}{2s\choose b}(-1)^b}.
\end{align*}
But then, since $2J_1^{2n}=I_1^{2n}$ (see the proof of lemma \ref{lemma:integral}),
\begin{align*}
2\tilde{f}_w=&2u_w+I_{1}^{2n}\sum\limits_{\substack{0\leq\alpha<\beta \\ \alpha+\beta+2=2n+1}}{c_{\alpha,\beta}\left({\alpha\choose a}{\beta\choose c}(-1)^{a+\beta-c}-{\alpha\choose c}{\beta\choose a}(-1)^{a+\alpha-c}\right)}\\
&+\sum\limits_{\substack{p\in \mathbb{N}_0,s\in \mathbb{N}\\2s=b+a-p}}{I_1^{2s} c_{2s}u_{x^pyx^c}\left({2s\choose a}(-1)^a-{2s\choose b}(-1)^{b}\right)}\\
&+\sum\limits_{\substack{q \in \mathbb{N}_0, s\in \mathbb{N}\\ 2s=b+c-q}}{I_{1}^{2s}c_{2s}u_{x^ayx^q}{2s\choose b}(-1)^b}\\
=&0
\end{align*}
as this is exactly the equation satisfied by the $c_{\alpha,\beta}$ (equation \eqref{eq:cab}).

A word $w=x^ayx^byx^c$ of even length, i.e. $a+b+c+2=2n$, will be contained in the sum (coming from the product $(\tpsi\cdot \PhiKZ)(x,y)$)
\begin{align*}
&\underbrace{1}_{\tpsi}\cdot \underbrace{\sum\limits_{\substack{\tilde{w};deg(\tilde{w})=2\\\left|\tilde{w}\right|=2n}}{u_{\tilde{w}} \tilde{w}}}_{\PhiKZ}+\underbrace{\sum\limits_{\substack{l,m\\2l+2m+2=2n}}{J_2^{l,m}x_{2l+1}x_{2m+1}}}_{\tpsi}\cdot \underbrace{1}_{\PhiKZ}\\
&+\sum\limits_{\substack{s\\2s+p+q+2=2n}}\sum\limits_{\substack{p,q\\v=x^pyx^q}}{\left(\underbrace{J_1^{2s}c_{2s}\ad_x^{2s}(y)}_{\tpsi} \underbrace{u_v v}_{\PhiKZ}+[y,\underbrace{J_1^{2s}c_{2s}\ad_x^{2s}}_{\tpsi}]\partial_y(\underbrace{u_v v}_{\PhiKZ})\right)}.
\end{align*}
The product $x_{2l+1}x_{2m+1}$ is given by
\begin{align*}
x_{2l+1}x_{2m+1}=&c_{2l}c_{2m}\left(\ad_{x}^{2l}(y)\ad_{x}^{2m}(y)+\left[y,\ad_x^{2l}(y)\right]\partial_y(\ad_x^{2m}(y))\right)\\
=&c_{2l}c_{2m}\left(\ad_{x}^{2l}(y)\ad_{x}^{2m}(y)+\ad_x^{2m}(\left[y,\ad_x^{2l}(y)\right])\right)
\end{align*}
Using that
\begin{equation*}
\ad_x^{2l}(y)\ad_x^{2m}(y)=\sum\limits_{i=0}^{2l}{\sum\limits_{j=0}^{2m}{{2l\choose i}{2m\choose j}(-1)^{i+j}x^iyx^{2l-i+j}yx^{2m-j}}}
\end{equation*}
and
\begin{equation*}
\ad_x^{2m}(\left[y,\ad_x^{2l}(y)\right])=\sum\limits_{i=0}^{2m}{\sum\limits_{j=0}^{2l}{{2m \choose i}{2l\choose j}(-1)^{i+j}(x^iyx^jyx^{2m+2l-i-j}-x^{i+j}yx^{2l-j}yx^{2m-i})}},
\end{equation*}
we find, by picking out the word $w=x^ayx^byx^c$ in the above sum, that the contributing coefficients are given by,
\begin{align*}
\tilde{f}_w=&u_w+\sum\limits_{\substack{l,m\\2l+2m+2=2n}}J_2^{l,m}c_{2l}c_{2m}\biggl({2l\choose a}{2m\choose c}(-1)^{a-c}\\
&+{2m\choose a}{2l\choose b}(-1)^{a+b}-{2m\choose c}{2l\choose b}(-1)^{b+c}\biggr)\\
&+\sum\limits_{\substack{p\in \mathbb{N}_0,s\in \mathbb{N}\\2s=b+a-p}}{J_1^{2s} c_{2s}u_{x^pyx^c}\left({2s\choose a}(-1)^a-{2s\choose b}(-1)^{b}\right)}\\
&+\sum\limits_{\substack{q \in \mathbb{N}_0, s\in \mathbb{N}\\ 2s=b+c-q}}{J_1^{2s} c_{2s}u_{x^ayx^q}{2s\choose b}(-1)^b}.
\end{align*}

We simplify the expression above. For this, let us consider the product,
\begin{align*}
J_{1}^{2s}c_{2s}u_{x^pyx^c}=&\frac{1}{2}I_1^{2s}\cdot\frac{2(4s+1)!\zeta(2s+1)}{(2\pi i)^{2s+1}((2s)!)^2}\cdot \frac{-1}{(2\pi i)^{p+c+1}}\zeta(x^pyx^c)\\
&=\frac{1}{2}\frac{((2s)!)^2}{(4s+1)!}\cdot\frac{2(4s+1)!\zeta(2s+1)}{(2\pi i)^{2s+1}((2s)!)^2}\cdot \frac{-1}{(2\pi i)^{p+c+1}}(-1)^c\frac{(p+c)!}{p!c!}\zeta(p+c+1)\\
&=\frac{(-1)^{c+1}(p+c)!}{(2\pi i)^{2s+p+c+2}p!c!}\zeta(2s+1)\zeta(p+c+1)\\
&=\frac{(-1)^{c+1}(p+c)!}{(2\pi i)^{2n}p!c!}\zeta(2s+1)\zeta(p+c+1).
\end{align*}
Similarly,
\begin{equation*}
J_{1}^{2s}c_{2s}u_{x^ayx^q}=\frac{(-1)^{q+1}(a+q)!}{(2\pi i)^{2n}a!q!}\zeta(2s+1)\zeta(a+q+1).
\end{equation*}
Therefore,
\begin{equation}\label{eq:tfw}
\begin{aligned}
\tilde{f}_w=&u_w+\sum\limits_{\substack{l,m\\2l+2m+2=2n}}J_2^{l,m}c_{2l}c_{2m}\biggl({2l\choose a}{2m\choose c}(-1)^{a-c}\\
&+{2m\choose a}{2l\choose b}(-1)^{a+b}-{2m\choose c}{2l\choose b}(-1)^{b+c}\biggr)\\
&+\sum\limits_{\substack{p\in \mathbb{N}_0,s\in \mathbb{N}\\2s=b+a-p}}{\frac{(-1)^{c+1}(p+c)!}{(2\pi i)^{2n}p!c!}\zeta(2s+1)\zeta(p+c+1)\left({2s\choose a}(-1)^a-{2s\choose b}(-1)^{b}\right)}\\
&+\sum\limits_{\substack{q \in \mathbb{N}_0, s\in \mathbb{N}\\ 2s=b+c-q}}{\frac{(-1)^{q+1}(a+q)!}{(2\pi i)^{2n}a!q!}\zeta(2s+1)\zeta(a+q+1){2s\choose b}(-1)^b }.
\end{aligned}
\end{equation}

\end{proof}
\subsection{Proof of Theorem \ref{thm:AT}}
Finally, we are in a position to prove our second main result.
\begin{proof}[Proof of Theorem \ref{thm:AT}]
We simply use formula \eqref{eq:tfw} for $w=x^2yx^4y$ to get,
\begin{align*}
\tilde{f}_w=&u_w+\sum\limits_{\substack{l,m\\2l+2m+2=8}}J_2^{l,m}c_{2l}c_{2m}\bigg({2l\choose 2}{2m\choose 0}(-1)^{2}\\
&+{2m\choose 2}{2l\choose 4}(-1)^{6}-{2m\choose 0}{2l\choose 4}(-1)^{4}\bigg)\\
&+\sum\limits_{\substack{p\in \mathbb{N}_0,s\in \mathbb{N}\\2s=6-p}}{\frac{(-1)^{0+1}(p+0)!}{(2\pi i)^{8}p!0!}\zeta(2s+1)\zeta(p+0+1)\left({2s\choose 2}(-1)^2-{2s\choose 4}(-1)^{4}\right)}\\
&+\sum\limits_{\substack{q \in \mathbb{N}_0, s\in \mathbb{N}\\ 2s=4-q}}{\frac{(-1)^{q+1}(2+q)!}{(2\pi i)^{8}2!q!}\zeta(2s+1)\zeta(2+q+1){2s\choose 4}(-1)^4}.
\end{align*}
Expanding everything gives,
\begin{align*}
\tilde{f}_w=&u_w+J_2^{1,2}c_{2}c_{4}\bigg({2\choose 2}{4\choose 0}(-1)^{2}+{4\choose 2}\underbrace{{2\choose 4}}_{=0}(-1)^{6}-{4\choose 0}\underbrace{{2\choose 4}}_{=0}(-1)^{4}\bigg)\\
&+J_2^{2,1}c_{4}c_{2}\bigg({4\choose 2}{2\choose 0}(-1)^{2}+{2\choose 2}{4\choose 4}(-1)^{6}-{2\choose 0}{4\choose 4}(-1)^{4}\bigg)\\
&+\frac{(-1)4!}{(2\pi i)^{8}4!}\zeta(3)\zeta(5)\Bigg({2\choose 2}(-1)^2-\underbrace{{2\choose 4}}_{=0}(-1)^{4}\Bigg)\\
&+\frac{(-1)2!}{(2\pi i)^{8}2!}\zeta(5)\zeta(3)\left({4\choose 2}(-1)^2-{4\choose 4}(-1)^{4}\right)\\
&+\frac{(-1)0!}{(2\pi i)^{8}0!}\zeta(7)\underbrace{\zeta(1)}_{=0}\left({6\choose 2}(-1)^2-{6\choose 4}(-1)^{4}\right)\\
&+\frac{(-1)^{3}(2+2)!}{(2\pi i)^{8}2!2!}\zeta(3)\zeta(5)\underbrace{{2\choose 4}}_{=0}(-1)^4\\
&+\frac{(-1)2!}{(2\pi i)^{8}2!}\zeta(5)\zeta(3){4\choose 4}(-1)^4.
\end{align*}
Again, we mark the terms which equal 0 (note that we set $\zeta(1)=0$). This simplifies to
\begin{align*}
\tilde{f}_w=&u_w+J_2^{1,2}c_{2}c_{4}{2\choose 2}{4\choose 0}+J_2^{2,1}c_{4}c_{2}\bigg({4\choose 2}{2\choose 0}+{2\choose 2}{4\choose 4}-{2\choose 0}{4\choose 4}\bigg)\\
&+\frac{(-1)}{(2\pi i)^{8}}\zeta(3)\zeta(5){2\choose 2}+\frac{(-1)}{(2\pi i)^{8}}\zeta(5)\zeta(3)\left({4\choose 2}-{4\choose 4}\right)+\frac{(-1)}{(2\pi i)^{8}}\zeta(5)\zeta(3){4\choose 4}.
\end{align*}
Hence,
\begin{align*}
\tilde{f}_w=&u_w+J_2^{1,2}c_{2}c_{4}+J_2^{2,1}c_{4}c_{2}(6+1-1)\\
&+\frac{(-1)}{(2\pi i)^{8}}\zeta(3)\zeta(5)\\
&+\frac{(-1)}{(2\pi i)^{8}}\zeta(5)\zeta(3)\left(6-1\right)\\
&+\frac{(-1)}{(2\pi i)^{8}}\zeta(5)\zeta(3)\\
=&u_w+c_{2}c_{4}(J_2^{1,2}+6J_2^{2,1})-\frac{7}{(2\pi i)^{8}}\zeta(3)\zeta(5).
\end{align*}
Moreover,
\begin{align*}
J_2^{1,2}&=\frac{238}{51609600}\\
J_2^{2,1}&=\frac{1199}{154828800}\\
c_2&=\frac{2(4+1)!\zeta(3)}{(2\pi i)^3 ((2)!)^2}=\frac{60\zeta(3)}{(2\pi i)^3}\\
c_4&=\frac{2(8+1)!\zeta(5)}{(2\pi i)^5 ((4)!)^2}=\frac{1260\zeta(5)}{(2\pi i)^5}\\
u_{x^2yx^4y}&=\frac{1}{(2\pi i)^8}\zeta(3,5).
\end{align*}
Therefore,
\begin{align*}
\tilde{f}_w&=\frac{\zeta(3,5)}{(2\pi i)^8}+\frac{60\cdot 1260 \zeta(3)\zeta(5)}{(2\pi i)^8}\left(\frac{238}{51609600}+6\frac{1199}{154828800}\right)-\frac{7}{(2\pi i)^{8}}\zeta(3)\zeta(5)\\
&=\frac{\zeta(3,5)}{(2\pi i)^8}-\frac{6293}{2048}\frac{\zeta(3)\zeta(5)}{(2\pi i)^8}=\frac{2048\zeta(3,5)-6293\zeta(3)\zeta(5)}{2048\cdot256\pi^8}\\
&=\frac{2048\zeta(3,5)-6293\zeta(3)\zeta(5)}{524288\pi^8}.
\end{align*}

The proof of the irrationality of $\tilde{f}_w$ is analogous to the case of the coefficient $f_w$ in $\PhiOH$ (see Theorem \ref{thm:OH}).
\end{proof}

\bibliographystyle{plain}
\bibliography{biblio}

\begin{thebibliography}{10}

\bibitem{Alekseev2010}
Anton Alekseev and Charles Torossian.
\newblock {K}ontsevich deformation quantization and flat connections.
\newblock {\em Communications in Mathematical Physics}, 300(1):47--64, 2010.

\bibitem{Artin64}
E.~Artin.
\newblock {\em The Gamma function}.
\newblock New York: Holt, Rinehart, and Winston, 1964.

\bibitem{Broadhurst}
David Broadhurst.
\newblock Multiple zeta values and modular forms in quantum field theory.
\newblock In C.~Schneider and J.~Bluemlein, editors, {\em {Computer Algebra in
  Quantum Field Theory: Integration, Summation and Special Functions}}, Texts
  and Monographs in Symbolic Computation, pages 33--73. Springer, 2013.

\bibitem{Brown2012}
Francis Brown.
\newblock On the decomposition of motivic multiple zeta values.
\newblock {\em Advanced Studies in Pure Mathematics}, Galois Theory and
  Arithmetic Geometry(63):31--58, 2012.

\bibitem{Drinfeld1991}
Vladimir.~G. Drinfeld.
\newblock On quasitriangular quasi-{H}opf algebras and on a group that is
  closely connected with {${\rm Gal}(\overline{\bf Q}/{\bf Q})$}.
\newblock {\em Leningrad Math. J.}, 2(4):829--860, 1991.

\bibitem{Euler1735}
Leonard Euler.
\newblock De summis serierum reciprocarum.
\newblock {\em Commentarii academiae scientiarum Petropolitanae}, 7:123--134,
  1735.

\bibitem{Euler1775}
Leonard Euler.
\newblock Meditationes circa singulare serierum genus.
\newblock {\em Novi Commentarii academiae scientiarum Petropolitanae},
  20:140--186, 1775.

\bibitem{Furusho}
Hidekazu {Furusho}.
\newblock {Double shuffle relation for associators}.
\newblock {\em Annals of Mathematics}, 174(1):341--360, 2011.

\bibitem{Le1996}
Thang Tu~Quoc Le and Jun Murakami.
\newblock {K}ontsevich's integral for the {K}auffman polynomial.
\newblock {\em Nagoya Math. J.}, 142:39--65, 1996.

\bibitem{Reutenauer}
Christophe Reutenauer.
\newblock {\em Free lie algebras}.
\newblock London Mathematical Society monographs. Clarendon Press New York,
  Oxford, 1993.

\bibitem{Rossi2014}
C.~A. {Rossi} and T.~{Willwacher}.
\newblock {P. Etingof's conjecture about Drinfeld associators}.
\newblock April 2014.
\newblock arXiv:1404.2047.

\end{thebibliography}

\end{document}